\documentclass[11pt,reqno]{amsart}
\pdfoutput=1
\usepackage{etex}
\usepackage{graphicx,mathrsfs,verbatim,amssymb,lineno,titletoc,enumerate,color}
\usepackage{pictex}
\usepackage{bbm}
\usepackage{caption}

\allowdisplaybreaks

\newtheorem{theorem}{Theorem}[section]

\newtheorem{lemma}[theorem]{Lemma}
\newtheorem{proposition}[theorem]{Proposition}

\newenvironment{acknowledgement}[1][Acknowledgement]{\textbf{#1.} }{ }

\numberwithin{counter}{section}

\def\N{\mathbb{N}}
\def\Z{\mathbb{Z}}
\def\Q{\mathbb{Q}}
\def\R{\mathbb{R}}
\def\E{\mathbb{E}}
\def\P{\mathbb{P}}

\def\eps{\varepsilon}

\newcommand{\eqand}{ \enskip \text{ and } \enskip }
\providecommand{\abs}[1]{\left\lvert#1\right\rvert}
\providecommand{\norm}[1]{\left\lVert#1\right\rVert_{\infty}}
\providecommand{\tv}[1]{\left\lVert#1\right\rVert_{TV}}

\DeclareMathOperator{\Poi}{\textnormal{Poi}}
\DeclareMathOperator{\diam}{\mathsf{diam}}

\usepackage[square,sort,comma,numbers]{natbib}
\usepackage[colorlinks]{hyperref}
\hypersetup{colorlinks=true, linkcolor=red, citecolor=green, urlcolor=magenta}

\begin{document}

\title{Poisson Statistics of Eigenvalues in the Hierarchical Dyson Model}

\author{Alexander Bendikov}
\thanks{The first author is supported by NCN Grant DEC-2012/05/B/ST 1/00613 of the Polish National
Center of Sciences.}
\address{Alexander Bendikov, Institute of Mathematics, Wroclaw University, Wroclaw, Poland.}
\email{\texttt{bendikov@math.uni.wroc.pl}}
 
\author{Anton Braverman}
\address{Anton Braverman, School of Operations Research and Information Engineering, Cornell University, Ithaca, NY.}
\email{\texttt{ab2329@cornell.edu}}
 
\author{John Pike}
\thanks{The third author is supported in part by NSF grant DMS-0739164.}
\address{John Pike, Department of Mathematics, Cornell University, Ithaca, NY.}
\email{\texttt{jpike@cornell.edu}}

\date{DRAFT of \today}

\begin{abstract}
Let $(X,d)$ be a locally compact separable ultrametric space. Given a
measure $m$ on $X$ and a function $C$ defined on the set $\mathcal{B}$ of
all balls $B\subset X$ we consider the hierarchical Laplacian $L=L_{C}$.
The operator $L$ acts in $L^{2}(X,m)$, is essentially self-adjoint, and has a
purely point spectrum. Choosing a family $\{\eps(B)\}_{B\in \mathcal{B}}$ 
of i.i.d. random variables, we define the perturbed function 
$\mathcal{C}(B)=C(B)(1+\eps(B))$ and the perturbed hierarchical Laplacian 
$\mathcal{L}=L_{\mathcal{C}}$. All outcomes of the perturbed operator $\mathcal{L}$ 
are hierarchical Laplacians. In particular they all have purely
point spectrum. We study the empirical point process $M$ defined in terms of 
$\mathcal{L}$-eigenvalues. Under some natural assumptions $M$ can be approximated 
by a Poisson point process. Using a result of Arratia, Goldstein, 
and Gordon based on the Chen-Stein method, we provide total variation convergence 
rates for the Poisson approximation. We apply our theory to random
perturbations of the operator $\mathfrak{D}^{\alpha }$, the $p$-adic
fractional derivative of order $\alpha >0$. This operator, related to the
concept of $p$-adic Quantum Mechanics, is a hierarchical Laplacian which acts in $L^{2}(X,m)$ 
where $X=\Q_{p}$ is the field of $p$-adic numbers and $m$ is Haar measure. 
It is translation invariant and the set $\mathsf{Spec}(\mathfrak{D}^{\alpha })$ 
consists of eigenvalues $p^{\alpha k}$, $k\in \Z$, each of which has infinite multiplicity.
\end{abstract}

\maketitle

\setcounter{tocdepth}{1}
\tableofcontents

\section{Introduction}
\label{sec:intro}
\setcounter{equation}{0}
The concept of hierarchical lattice and hierarchical distance was proposed by 
F.J. Dyson in his famous papers on the phase transition for a  one-dimensional
ferromagnetic model with long range interaction \cite{Dyson1,Dyson2}. 
The notion of the hierarchical Laplacian $L$, which is closely related to 
Dyson's model, was studied in several mathematical papers 
\cite{Bovier,Kvitchevski1,Kvitchevski2,Kvitchevski3,Molchanov,AlbeverioKarwowski,PearsonBellisard,Kozyrev}. 
These papers contain some basic information about $L$ 
(the spectrum, Markov semigroup, resolvent, etc...) 
in the case when the state space $(X,d,m)$ is discrete and 
satisfies some symmetry conditions such as homogeneity and self-similarity. 
As was noticed in \cite{Figa-Tal1,Figa-Tal2}, these symmetry
conditions imply that $(X,d,m)$ can be identified with some
discrete infinitely generated Abelian group $G$ equipped with a translation
invariant ultrametric and measure. The Markov semigroup $\{P^{t}=\exp (-tL)\}_{t \geq 0}$
acting on $L^{2}(G,m)$ becomes symmetric, translation invariant, and
isotropic. In particular, $\mathsf{Spec}(L)$ is pure point and all
eigenvalues have infinite multiplicity.

The main goal of the papers mentioned above was to study the corresponding
Anderson Hamiltonian $H=L+V$; the hierarchical Laplacian $L$ plus random
potential $V$. There was a hope to detect for such operators the spectral
bifurcation from the pure point spectrum to the continuous one, i.e. to
justify the famous Anderson conjecture. Unfortunately, the opposite result was true. 
Under mild technical conditions, the hierarchical Anderson
Hamiltonian has a pure point spectrum\textemdash the phenomenon of localization 
(see \cite{PearsonBellisard,Krutikov1,Krutikov2,Kvitchevski2}). 
Moreover, the local statistics of the spectrum of $H$ are Poissonian \cite{Kvitchevski3}, 
which is always deemed a manifestation of the spectral localization \cite{AizMol,Minami}.

We study a class of operators introduced in \cite{BGMS}, the random
hierarchical Laplacians, which demonstrate several new
spectral effects. The spectra of such operators is still pure point (with
compactly supported eigenfunctions), but in contrast to the deterministic
case there exists a continuous density of states. This density detects the
spectral bifurcation from the pure point spectrum to the continuous one. The
eigenvalues \emph{form locally} a Poisson point process with intensity
given by the density of states. That is, the empirical point process defined in
terms of the eigenvalues is approximated by a Poisson point
process. In this paper we provide an error bound for the Poisson approximation 
in terms of the total variation distance; see Theorem~\ref{thm: main}. 
We prove Theorem~\ref{thm: main} by applying a result due to 
R. Arratia, L. Goldstein and L. Gordon in \cite{AGG}, 
which studies Poisson approximations using the Chen-Stein method.

Throughout the paper we assume that $(X,d)$ is a locally compact, non-compact, and
separable ultrametric space. Recall that a metric $d$ is called an 
\emph{ultrametric} if it satisfies the ultrametric inequality 
\begin{equation}
d(x,y)\leq \max \{d(x,z),d(z,y)\},
\end{equation}
which is stronger than the usual triangle inequality. We also assume that the 
ultrametric $d$ is \emph{proper}, i.e.\ each closed $d$-ball is a compact set.

Let $m$ be a Radon measure on $(X,d)$ such that
\begin{itemize}
\item $m(B)>0$ for each ball $B$ of positive diameter.

\item $m(\{x\})=0$ if and only if $x$ is not an isolated point.
\end{itemize}
Let $\mathcal{B}$ be the set of all balls having positive measure. Our
assumptions imply that the set $\mathcal{B}$ is countable. Let 
$C:\mathcal{B}\rightarrow (0,\infty )$ be a function which satisfies the
following assumptions:
\begin{equation*}
\lambda (B):=\sum\limits_{T\in \mathcal{B}:\text{ }B\subseteq T}C(T)<\infty
\end{equation*}
and
\begin{equation*}
\sup \{\lambda (B):B\in \mathcal{B}\text{ and }x\in B\}=\infty
\end{equation*}%
holds for all $B\in \mathcal{B}$ and non-isolated $x\in X$. As was shown
in \cite{BendikovKrupski}, the set of functions $C$ which satisfy the above
assumptions is quite rich.

Let $\mathcal{D}$ be the set of all locally constant functions having
compact support. Given the data $(X,m,C)$ one defines (pointwise) the
hierarchical Laplacian 
\begin{equation}
L_{C}f(x):=\sum\limits_{B\in \mathcal{B}:\text{ }x\in B}C(B)
\left( f(x)-\frac{1}{m(B)}\int_{B}fdm\right),\quad f\in \mathcal{D}.
\label{hlaplacian}
\end{equation}
The operator $(L_{C},\mathcal{D})$ acts in $L^{2}=L^{2}(X,m)$, is symmetric,
and admits a complete system of eigenfunctions $\{f_{B,B^{\prime }}\}$, defined as
\begin{equation}
f_{B,B^{\prime }}
=\frac{1}{m(B)}\mathbf{1}_{B}-\frac{1}{m(B^{\prime })}\mathbf{1}_{B^{\prime }},
\label{eigenfunction}
\end{equation}
where $B\subset B^{\prime }$ run over all nearest neighboring balls in $\mathcal{B}$. 
The family $\{f_{B,B^{\prime }}\}$ is called the Haar system associated to $(X,d,m)$. 
The eigenvalue $\lambda (B^{\prime })$ corresponding to the Haar function $f_{B,B^{\prime }}$ 
is computed as
\begin{equation}
\lambda (B^{\prime })=\sum\limits_{T\in \mathcal{B}:\text{ }B^{\prime}\subseteq T}C(T).  
\label{eigenvalue}
\end{equation}
Since the Haar system $\{f_{B,B^{\prime }}\}$ is complete, we conclude that 
$(L_{C},\mathcal{D})$ is an essentially self-adjoint operator in $L^{2}$. By a slight abuse of notation, 
we write $(L_{C},\mathsf{Dom}_{L_{C}})$ for its unique self-adjoint extension. 
For a more detailed discussion we refer the reader to \cite{BendikovKrupski}. 
See also the related papers \cite{BGP} and \cite{BGPW}.

Observe that to define the functions $C(B)$, $\lambda (B)$, and therefore the
operator $(L_{C},\mathsf{Dom}_{L_{C}})$, we do not need to specify the
ultrametric $d$. What is needed is the family of balls $\mathcal{B}$, which
can be identical for two different ultrametrics $d$ and $d^{\prime}$. 
In particular, given the data $(X,\mathcal{B},m)$, and choosing the function 
\begin{equation*}
C(B)=\frac{1}{m(B)}-\frac{1}{m(B^{\prime })},
\end{equation*}
where $B\subset B^{\prime }$ are any two nearest neighbor balls in 
$\mathcal{B}$, we obtain the hierarchical Laplacian 
$(L_{C},\mathsf{Dom}_{L_{C}})$ satisfying 
\begin{equation*}
\lambda (B)=\frac{1}{m(B)}.
\end{equation*}
We refer to $(L_{C},\mathsf{Dom}_{L_{C}})$ as the standard
hierarchical Laplacian associated with the data $(X,\mathcal{B},m)$.

\subsection*{Notation}
\label{subsec:notation}
$ $

Denoting the positive integers by $\N$, we write $\N_{0}=\N\cup\{0\}$ and $\N_{\geq 2}=\N\setminus\{1\}$. 
For a function $f: \R \to \R$, define $\norm{f} = \sup_{x \in \R} \abs{f(x)}$. 
For a real-valued random variable $X$, we let $\mathscr{L}(X)$ denote the law of that random variable. 
The law of a Poisson random variable with mean $\lambda$ is denoted $\Poi(\lambda)$. 
The \emph{total variation distance} between probability measures $\mu$ and $\nu$ on $\R$ 
with the Borel $\sigma$-field $\mathscr{B}(\R)$ is defined as 
\[
\tv{\mu - \nu}=\sup_{A\in\mathscr{B}(\R)}\abs{\mu(A)-\nu(A)}.
\]
When $\mu$ and $\nu$ have countable support $S$, total variation is half the $L^{1}$ distance between the 
associated mass functions:
\[
\tv{\mu - \nu}=\frac{1}{2}\sum_{x\in S}\abs{\mu\left(\{x\}\right) - \nu\left(\{x\}\right)}.
\]

\subsection*{Outline}
\label{subsec:outline}
$ $

In Section~\ref{sec:homogeneous laplacian} we specify some spectral properties of the 
hierarchical Laplacian $L_{C}$ assuming that the ultrametric measure space where it acts 
and the Laplacian itself both satisfy certain symmetry conditions (homogeneity and self-similarity).
As an example, we consider the operator $\mathfrak{D}^{\alpha }$ of the 
$p$-adic fractional derivative of order $\alpha >0$. This operator is related to
the concept of $p$-adic Quantum Mechanics and was introduced by V.S. Vladimirov; see \cite{Vladimirov}, 
\cite{VladimirovVolovich} and \cite{Vladimirov94}. $\mathfrak{D}^{\alpha }$ 
is a homogeneous hierarchical Laplacian which acts in $L^{2}(X,m)$ where $X=\Q_{p}$ is the field of 
$p$-adic numbers and $m$ is its Haar measure. It is translation invariant and the set 
$\mathsf{Spec}(\mathfrak{D}^{\alpha })$ consists of the eigenvalues 
$\{p^{\alpha k}\}_{k\in\Z}$, each having infinite multiplicity, and contains $0$ as 
an accumulation point.

In Section~\ref{sec:random perturbations}, given a homogeneous Laplacian $L_{C}$ and a family 
$\{\eps (B)\}_{B\in \mathcal{B}}$ of symmetric i.i.d. random variables, we define a
perturbed function $\mathcal{C}(B)=C(B)\left(1+\eps(B)\right)$ and a perturbed Laplacian 
$\mathcal{L}=L_{\mathcal{C}}$. For each $\omega$, the operator 
$\mathcal{L}(\omega)=L_{\mathcal{C}(\omega)}$ is a hierarchical Laplacian, 
whence it has a pure point spectrum. On the other hand, for some $\omega$ it may fail to be homogeneous. 
In particular, the set of its eigenvalues may form dense subsets in certain intervals \cite{BendikovKrupski}. 
Following \cite{BGMS} we recall some basic notions and properties associated to the perturbed 
operator $\mathcal{L}$.

Our main result appears in Section~\ref{sec:poisson convergence}. 
We fix a horocycle $H$ (the set of all balls having the same diameter) and 
define the empirical process 
\begin{equation*}
M_{O}=\sum_{B\subseteq O:B\in H}\delta _{\lambda(B)}, \text{ }O\in \mathcal{B},
\end{equation*}
associated with eigenvalues $\lambda(B)$ of the perturbed operator $\mathcal{L}$.
Here $\delta_{a}$ denotes the point mass at $a$, so that for any interval $I$, 
$\delta_{\lambda(B)}(I)=1\left\{\lambda(B)\in I\right\}$ is the indicator that $\lambda(B)$ belongs 
to $I$ and $M_{O}(I)$ is the number of $\lambda(B)$, $B\subseteq O$, which fall in $I$.
Under mild assumptions, it was shown in \cite{BGMS} that for an appropriate sequence of intervals 
$I_{O}$, when $\E[M_{O}(I_{O})]$ converges to some value $\lambda>0$ as $O\rightarrow\varpi$, 
the random variable $M_{O}(I_{O})$ converges in law to a Poisson random variable $\Poi(\lambda)$. 
In Theorem~\ref{thm: main}, we give an upper bound on the total variation distance 
\begin{align*}
\tv{\mathscr{L}(M_{O}(I_{O}))- \Poi(\lambda)}.
\end{align*}
Random perturbations of the operator $\mathfrak{D}^{\alpha }$ 
of the $p$-adic fractional derivative of order $\alpha >1$ from Section~\ref{sec:homogeneous laplacian} 
provide an example where our result can be applied.

Section~\ref{sec:bounds_proof} provides a proof of Theorem~\ref{thm: main}. The proof is based on a 
result by R. Arratia, L. Goldstein, and L. Gordon \cite{AGG}, who study Poisson approximation of random 
variables using the Chen-Stein method.

$ $

\begin{acknowledgement}
This work was initiated at Cornell University. We are grateful to Larry Goldstein, Michael
Nussbaum, and Laurent Saloff-Coste for fruitful discussions and valuable comments. 
We also thank the organizers of the 2015 Workshop on New Directions in Stein's Method 
where part of this research was carried out.
\end{acknowledgement}

\section{Homogeneous Laplacian}
\label{sec:homogeneous laplacian}
\setcounter{equation}{0}

In this section we specify some spectral properties of the hierarchical Laplacian 
$L_{C}$ assuming that the (non-compact) ultrametric measure space $(X,d,m)$ 
where it acts and the Laplacian by itself satisfy the following symmetry conditions:
\begin{itemize}
\item The group of isometries of $(X,d)$ acts transitively on $X$.

\item Both the reference measure $m$ and the function $C(B)$ are
invariant under the action of isometries.
\end{itemize}
When the ultrametric measure space $(X,d,m)$ and hierarchical Laplacian 
$L_{C}$ satisfy both of these conditions, we call them \emph{homogeneous}.
The first assumption implies that $(X,d)$ is either discrete or perfect.
Basic examples that we have in mind are
\begin{itemize}
\item $X=\Z_{p}$ -- the compact group of $p$-adic integers.

\item $X=\Q_{p}$ -- the ring of $p$-adic numbers.\footnote{Though we sometimes refer to the field 
$\Q_{p}$, the assumption that $p$ is prime is never needed in this paper.}

\item $X=\Q_{p}/\Z_{p}\cong \Z(p^{\infty })$ --
the multiplicative group of $p^{n}$th roots of unity, $n\in\N$, considered in the discrete topology.

\item $X=S_{\infty}$ -- the infinite symmetric group.
\end{itemize}

As was noticed in \cite{Figa-Tal1} and \cite{Figa-Tal2}, our assumptions
imply that the measure space $(X,m)$ can be identified with a locally
compact totally disconnected Abelian group $G$ equipped with its Haar
measure. Notice that the group $G$ is not unique. As a possible choice of
$G$ when, for instance, $X$ is \emph{perfect}, one can take the Abelian group
\begin{equation}
G=\mathop{\mathrm{limind}}_{\ell\rightarrow -\infty}
\left(\prod\limits_{k\geq \ell}\Z(n_{k})\right),  
\label{group G-X}
\end{equation}
where $\Z(n_{k})$ are (nontrivial) cyclic groups and $\{n_{k}\}_{k\in \Z}$ 
is an appropriately chosen double sided sequence of integers. The canonical ultrametric 
structure on $G$ is defined by the descending sequence of compact subgroups
\begin{equation*}
G_{\ell}=\prod\limits_{k\geq \ell}\Z(n_{k}).
\end{equation*}
Namely, the groups $G_{\ell}$ and their cosets $\{a+G_{\ell}\}$
form the collection $\mathcal{B}$ of all clopen balls.

There is a natural ultrametric structure associated to the chain of (small)
subgroups $G_{\ell}$ of $G$. One defines the \emph{absolute value} 
$\left\vert a\right\vert$ for the elements $a$ of $G$,
\begin{equation*}
\left\vert a\right\vert =\left\{ 
\begin{array}{ccc}
0 & \text{if} & a=0, \\ 
m(G_{\ell}) & \text{if} & a\in G_{\ell}\setminus G_{\ell+1}.
\end{array}
\right.
\end{equation*}
The absolute value $\left\vert a\right\vert$ satisfies the ultrametric inequality
\begin{equation*}
\left\vert a+b\right\vert \leq \max \{\left\vert a\right\vert ,\left\vert b\right\vert \}.
\end{equation*}
It is also clear that $\left\vert a\right\vert =\left\vert -a\right\vert$
and that $d(a,b):=\left\vert a-b\right\vert $ is an ultrametric that gives 
$(G,m)$ the structure of a homogeneous ultrametric measure space as defined
above. In particular, for any ball $B$ we have 
\begin{equation*}
m(B)=\diam(B).
\end{equation*}
Choosing the Haar measure $m$ such that $m(G_{0})=1$, we compute $m(G_{\ell})$
for any $\ell\neq 0$ as follows:
\begin{equation*}
m(G_{\ell})=\left\{ 
\begin{array}{ccc}
n_{\ell}...n_{-1} & \text{if} & \ell<0, \\ 
\left( n_{\ell-1}...n_{0}\right) ^{-1} & \text{if} & \ell>0.
\end{array}
\right. 
\end{equation*}
Recall that in the classical setting $X=\Q_{p}$, we chose 
$G=\bigcup_{\ell\in \Z}G_{\ell}$, $G_{\ell}=p^{\ell}\Z_{p}$, and set 
\begin{equation*}
\left\vert a\right\vert =p^{-n(a)},\text{ where }n(a)=\max \{\ell:a\in G_{\ell}\}.
\end{equation*}
In this case, the quantity $\left\vert a\right\vert $ becomes a pseudonorm:  
$\left\vert ab\right\vert \leq \left\vert a\right\vert \left\vert b\right\vert$.
It is a norm ($\left\vert ab\right\vert =\left\vert a\right\vert\left\vert b\right\vert$) 
if $p$ is a prime number\textemdash the basic property in $p$-adic analysis and its applications. 

We recall that to any given ultrametric space $(X,d)$ one associates in a
standard fashion a tree $\mathcal{T}(X)$ (see Figure~\ref{fig5}). The vertices of the
tree are metric balls, and thus in our case they are the cosets 
$\{a+G_{\ell}:a\in G,\ell\in \Z\}$. The ascending sequence of subgroups 
$\{G_{\ell}:l=0,-1,-2,...\}$ identifies a special boundary point, which we
denote $\varpi$. With respect to this special point we consider the
horocycles of the tree. A \emph{horocycle} in this case is the set of
vertices consisting of the balls of a given diameter; in other words, the
cosets relative to the same subgroup $G_{\ell}$. Thus, for fixed $\ell$, the
horocycle is $H_{\ell}=\{a+G_{\ell}:a\in G\}$. The boundary $\partial \mathcal{T}(G)$
can be identified with the one-point compactification $G\cup \{\varpi\}$ of 
$G$. We refer to \cite{Figa-Tal1}, \cite{Figa-Tal2}, and \cite{BendikovKrupski} 
for a complete treatment of the association between an
ultrametric space and the tree of its metric balls. The most complete source
for the basic definitions related to the geometry of trees is \cite{Cartier1972}; see also \cite{Woess}.

\begin{figure}
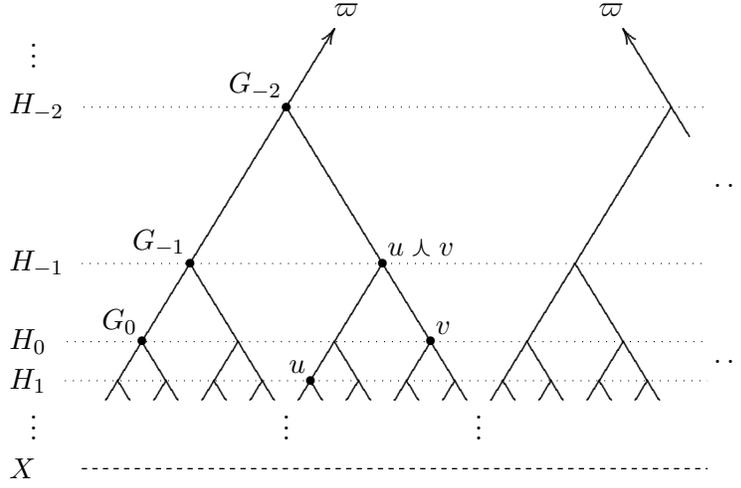

$$
\beginpicture

\setcoordinatesystem units <.8mm,1.3mm>

\setplotarea x from -10 to 104, y from -5 to 46

\arrow <6pt> [.2,.67] from 2 2 to 40 40

\plot 32 32 62 2 /

 \plot 16 16 30 2 /

 \plot 48 16 34 2 /

 \plot 8 8 14 2 /

 \plot 24 8 18 2 /

 \plot 40 8 46 2 /

 \plot 56 8 50 2 /

 \plot 4 4 6 2 /

 \plot 12 4 10 2 /

 \plot 20 4 22 2 /

 \plot 28 4 26 2 /

 \plot 36 4 38 2 /

 \plot 44 4 42 2 /

 \plot 52 4 54 2 /

 \plot 60 4 58 2 /

\arrow <6pt> [.2,.67] from 99 29 to 88 40

 \plot 66 2 96 32 /

 \plot 70 2 68 4 /

 \plot 74 2 76 4 /

 \plot 78 2 72 8 /

 \plot 82 2 88 8 /

 \plot 86 2 84 4 /

 \plot 90 2 92 4 /

 \plot 94 2 80 16 /

\setdots <3pt>
\putrule from -4.8 4 to 102 4
\putrule from -4.5 8 to 102 8
\putrule from -2 16 to 102 16
\putrule from -1.7 32 to 102 32

\setdashes <2pt> \linethickness=.5pt
\putrule from -2 -5 to 102 -5

\put {$\vdots$} at 32 0
\put {$\vdots$} at 64 0

\put {$\dots$} [l] at 103 6
\put {$\dots$} [l] at 103 24

\put {$H_{-2}$} [l] at -14 32
\put {$H_{-1}$} [l] at -14 16
\put {$H_0$} [l] at -14 8
\put {$H_1$} [l] at -14 4
\put {$X$} [l] at -14 -5
\put {$\vdots$} at -10 0
\put {$\vdots$} [B] at -10 36

\put {$\scriptstyle\bullet$} at 8 8
\put {$G_0$} [rb] at 7.2 8.8
\put {$\scriptstyle\bullet$} at 16 16
\put {$G_{-1}$} [rb] at 15.2 16.8
\put {$\scriptstyle\bullet$} at 32 32
\put {$G_{-2}$} [rb] at 31.2 32.8
\put {$\scriptstyle\bullet$} at 36 4 
\put {$u$} [rb] at 35.2 4.8
\put {$\scriptstyle\bullet$} at 56 8 
\put {$v$} [lb] at 56.8 8.8
\put {$\scriptstyle\bullet$} at 48 16 
\put {$u \curlywedge v$} [lb] at 48.8 16.8

\put {$\varpi$} at 42 42
\put {$\varpi$} at 86 42

\endpicture
$$
\caption{Tree of balls $\mathcal{T}(X)$ with forward degree $n_l=2$.}
\label{fig5}
\end{figure}

Let $L_{C}$ be a homogeneous hierarchical Laplacian.
Thanks to the homogeneity property, $C(A)=C(B)$ for any two balls which
belong to the same horocycle $H$. The same of course is true for the
eigenvalues $\lambda (A)$ and $\lambda (B)$. We set $C_{H}=C(B)$ and 
$\lambda _{H}=\lambda (B)$ for any ball $B\in H$. When $H=H_{k}$ we will
also write $c_{k}=C_{H_{k}}$ and $\lambda _{k}=\lambda _{H_{k}}$. 
In this notation 
\begin{equation*}
\lambda _{k}=\sum\limits_{\ell\leq k}c_{\ell}.
\end{equation*}
Each ball $B\in H_{k}$ has $n_{k}$ nearest neighbors $B_{i}\subset B$.
The system of eigenfunctions $\{f_{B_{i},B}:i=1,...,n_{k}\}$ corresponding
to $\lambda (B)$ is not orthogonal: Its linear span $\mathcal{H}(B)\subset L^{2}$ 
has dimension $n_{k}-1$. For two different balls $S$ and $T$, the
eigenspaces $\mathcal{H}(S)$ and $\mathcal{H}(T)$ are orthogonal. It follows
that the eigenspace $\mathcal{H}_{k}$ corresponding to the horocycle $H_{k}$
(equivalently, to the eigenvalue $\lambda _{k}$) is of the form
\begin{equation*}
\mathcal{H}_{k}=\bigoplus_{B\in H_{k}}\mathcal{H}(B).
\end{equation*}
The system of eigenfunctions $\{f_{B_{i},B}\}_{B\in \mathcal{B}}$ is complete, whence
\begin{equation*}
\bigoplus_{k\in \Z}\mathcal{H}_{k}=L^{2}(G,m).
\end{equation*}

Among the variety of homogeneous hierarchical Laplacians $L_{C}$ on $(G,d,m)$,
we would like to mention a one-parameter family $\{\mathfrak{B}^{\alpha}\}_{\alpha >0}$. 
The hierarchical Laplacian $\mathfrak{B}^{\alpha }$ is defined by the function
\begin{equation}
C^{\alpha }(B)=\left(m(B)\right) ^{-\alpha }-\left(m(B^{\prime })\right)^{-\alpha },  
\label{C-B-fractur}
\end{equation}
where $B\subset B^{\prime }$ are nearest neighbor balls. Thus for any
$B\in \mathcal{B}$, the eigenvalue $\lambda^{\alpha}(B)$ of 
$\mathfrak{B}^{\alpha }$ corresponding to $B$ is
\begin{equation*}
\lambda ^{\alpha}(B)=\left(m(B)\right)^{-\alpha}.
\end{equation*}
We recall from \cite{BGPW} that the set $\mathcal{D}$ of compactly supported
locally constant functions is in the domain of the operator $\mathfrak{B}^{\alpha}$. 
It is remarkable, although not difficult to prove, that the following properties hold: 
\begin{equation*}
\mathfrak{B}^{\beta }:\mathcal{D}\rightarrow\mathsf{Dom}(\mathfrak{B}^{\alpha }),
\end{equation*}
and on $\mathcal{D}$,
\begin{equation*}
\mathfrak{B}^{\alpha }\circ \mathfrak{B}^{\beta }=\mathfrak{B}^{\alpha+\beta }
\text{ and }(\mathfrak{B}^{\alpha })^{\beta }=\mathfrak{B}^{\alpha\beta }.
\end{equation*}
Moreover, when $G=\Q_{p}$, the operator 
$\mathfrak{D}^{\alpha}=p^{\alpha}\mathfrak{B}^{\alpha}$ 
is the fractional derivative operator of order $\alpha$ as defined and studied via Fourier transform
in \cite{Taibleson75}, \cite{Vladimirov}, \cite{Vladimirov94} and \cite{Kochubey20004}:
\begin{equation*}
\widehat{\mathfrak{D}^{\alpha}f}(\xi)
=\left\vert \xi \right\vert ^{\alpha}\widehat{f}(\xi),\text{ \ }\xi \in \Q_{p},
\end{equation*}
and 
\begin{equation}
\mathfrak{D}^{\alpha}f(x)=\frac{p^{\alpha}-1}{1-p^{-\alpha -1}}
\int_{G}\frac{f(x)-f(y)}{\left\vert x-y\right\vert ^{1+\alpha }}dm(y).
\label{frac-deriv}
\end{equation}

Note that similar identifications (based on cyclic groups $\Z(n)$
as building blocks) can be carried over when $(X,d)$
is infinite and discrete. 
For instance, the infinite (non-Abelian) symmetric group $S_{\infty}$
equipped with its canonical ultrametric structure defined by the family 
$\{S_{n}\}$ of its finite symmetric subgroups can be identified (as an 
ultrametric measure space) with the discrete Abelian group 
$G=\bigoplus_{\ell>1}\Z(\ell)$. The group $G$ is equipped with its
canonical ultrametric structure defined by the family $\{G_{n}\}$ of its
finite subgroups $G_{n}=\prod_{\ell=2}^{n}\Z(\ell)$.

\section{Random perturbations}
\label{sec:random perturbations}
\setcounter{equation}{0}

Let $(X,d,m)$ be a non-compact homogeneous ultrametric space. 
Let $L_{C}$ be the homogeneous hierarchical Laplacian acting on $X$ 
and defined by the function $C$. Let $\{\eps(B)\}_{B\in\mathcal{B}}$ 
be a family of symmetric i.i.d. random variables defined on the probability space 
$(\Omega,\mathcal{F},\P)$, and taking values in some small interval 
$[-\epsilon ,\epsilon ]\subset \left(-1,1\right)$. We define the perturbed function 
$\mathcal{C}$ and the perturbed hierarchical Laplacian $\mathcal{L}$ as 
\begin{equation*}
\mathcal{C}(B)=C(B)(1+\eps (B)),
\end{equation*}
and
\begin{equation*}
\mathcal{L}f(x)=\sum_{B\in \mathcal{B}:\text{ }x\in B}
\mathcal{C}(B)\left( f(x)-\frac{1}{m(B)}\int_{B}fdm\right).
\end{equation*}
Evidently, $\mathcal{L}(\omega)$ may well be non-homogeneous for some $\omega\in\Omega$. 
Still it has a pure point spectrum for all $\omega$, but the structure of the closed set 
$\mathsf{Spec}(\mathcal{L}(\omega))$ can be quite complicated. See \cite{BendikovKrupski} 
for various examples.

\subsection*{Two stationary families.}
\label{subsec:stationary families}
$ $

We study the eigenvalues of the perturbed operator. Because of the homogeneity assumption 
we can identify its eignevalues with a certain stationary family of random variables. 
This observation is crucial in our analysis.

Let us fix a horocycle $H=H_{\ell}$ for some $\ell\in\Z$ and let 
$\lambda_{H}=\lambda _{\ell}$ be the eigenvalue of the homogeneous Laplacian $L_{C}$
corresponding $H$. For a given $B\in H$, let $\{B_{k}\}_{k\leq\ell}$ be the unique 
geodesic path in $\mathcal{T}(X)$ from $\varpi$ to $B$. 
The eigenvalue $\lambda(B)$ of the perturbed operator $\mathcal{L}$ is
\begin{eqnarray*}
\lambda(B) &=&\sum_{k\leq \ell}\mathcal{C}(B_{k})
=\sum_{k\leq \ell}c_{k}\left( 1+\eps(B_{k})\right) \\
&=&\lambda _{\ell}\left( 1+\sum\limits_{k\leq \ell}\alpha_{k}\eps(B_{k})\right) 
:=\lambda _{\ell}\left( 1+U(B)\right),
\end{eqnarray*}
where $\alpha_{k}=c_{k}/\lambda _{\ell}$, and
\begin{equation}
U(B)=\sum_{k\leq \ell}\alpha_{k}\eps(B_{k}).
\label{U-H}
\end{equation}
Notice that $\sum_{k\leq \ell}\alpha_{k}=1$ and that $\{U(B)\}_{B\in H}$ are
(dependent) identically distributed symmetric random variables taking values
in some symmetric interval $I\subset \left(-1,1\right)$.

As the horocycle $H=H_{\ell}$ is fixed, it is useful to identify the balls $B\in H$ with 
elements $g\in G$ of the (discrete!) Abelian group 
$G=\bigoplus_{k<\ell}\Z(n_{k})$. With such an identification in mind 
it is now straightforward to show that the family of random variables 
$\{U(g)\}_{g\in G}=\{U(B)\}_{B\in H}$ (respectively, 
$\{\lambda(g)\}_{g\in G}=\{\lambda(B)\}_{B\in H}$) is stationary. 
That is, for any $g,g_{1},...,g_{s}$ in $G$, 
\begin{equation*}
\left\{U(g+g_1),\ldots,U(g+g_s)\right\}\overset{d}{=}\left\{U(g_1),\ldots,U(g_s)\right\},
\end{equation*}
respectively,
\begin{equation*}
\left\{\lambda(g+g_{1}),\ldots,\lambda(g+g_{s})\right\}\overset{d}{=}
\left\{\lambda(g_{1}),\ldots,\lambda(g_{s})\right\}.
\end{equation*}
For the general theory of stationary processes indexed by $\Z$ or $\R$ 
we refer to \cite{IbragimovLinnik} and \cite{LeadbetterLingremRootzen}.

\subsection*{The integrated density of states}
\label{subsec:ids}
$ $

Let $O\in\mathcal{B}$ be an ultrametric ball. 
We fix a horocycle $H$ and define the normalized empirical process 
\begin{equation}
\mathcal{M}_{HO}
=\frac{1}{\left\vert \mathcal{B}_{HO}\right\vert}\sum_{B\in\mathcal{B}_{HO}}\delta_{\lambda(B)},
\label{Empirical distribution}
\end{equation}
where $\mathcal{B}_{HO}=\{B\in \mathcal{B}:B\in H\text{ and }B\subset O\}$. As was
shown in \cite{BGMS}, there exists a probability measure $\mathcal{M}_{H}$ 
such that $\P$-a.s. 
\begin{equation}
\mathcal{M}_{HO}\rightarrow \mathcal{M}_{H}\text{ weakly as } O\rightarrow \varpi. 
\label{M-H-measure}
\end{equation}

The measure $\mathcal{M}_{H}$ is called the \emph{integrated density of states} 
(corresponding to the horocycle $H$). If there exists a measurable function
 $\mathfrak{m}_{H}:\R\rightarrow\R$ such that for any interval $I$
\begin{equation}
\mathcal{M}_{H}(I)=\int_{I}\mathfrak{m}_{H}(\tau )d\tau , 
\label{def=ids}
\end{equation}
the function $\mathfrak{m}_{H}(\tau )$ is called the density of states 
(corresponding to the horocycle $H$).
The question of whether $\mathfrak{m}_{H}$ exists, is
continuous, belongs to the class $C^{\infty }$, etc... is basic in various applications; 
see for instance \eqref{eq:poisconv} in Section~\ref{sec:poisson convergence}. 

Recall that for any $B\in H$,
\begin{equation*}
\lambda(B)=\lambda _{H}\left(1+U(B)\right),
\end{equation*}
where the $U(B)$ are identically distributed (dependent) symmetric random variables. 
Let $\mathcal{N}_{H}$ denote their common distribution. 
It turns out that the normalized empirical process 
\begin{equation*}
\mathcal{N}_{HO}=\frac{1}{\left\vert \mathcal{B}_{HO}\right\vert}
\sum_{B\in \mathcal{B}_{HO}}\delta _{U(B)}
\end{equation*}
converges to $\mathcal{N}_{H}$ weakly $\P$-a.s. Evidently, the probability measures 
$\mathcal{M}_{H}$ and $\mathcal{N}_{H}$ are related by the affine transformation 
\begin{equation*}
\mathcal{N}_{H}=\mathcal{M}_{H}\circ \vartheta,\text{ \ }
\vartheta :\tau\rightarrow \lambda _{H}\tau +\lambda _{H}.
\end{equation*}
In particular, $\mathcal{M}_{H}$ is absolutely continuous w.r.t. Lebesgue
measure if and only if $\mathcal{N}_{H}$ is. The measure $\mathcal{N}_{H}$
has a remarkable feature -- it belongs to the class $\mathfrak{J}$ of probability measures 
which can be represented as the distribution of some random variable $U$ of the form
\begin{equation}
U=\sum_{k\geq 0}b_{k}\eps_{k},  
\label{Class I}
\end{equation}
where $\{\eps_{k}\}_{k\geq 0}$ are symmetric i.i.d. random variables and 
$b_{k}>0$ satisfies $\sum b_{k}=1$.

Various properties of $\mathfrak{J}$-distributions (infinite convolutions)
have been studied by many authors since the 1930's. See e.g. \cite{Lukacs}, 
\cite{GrahamMcGehee}, \cite{Solomyak}, \cite{PersSolomyak}, 
\cite{PersSchlagSolomyak}, \cite{Solomyak1} and references therein. 
We would like to mention three remarkable properties of $\mathfrak{J}$-distributions. 
The first one is due to L\'{e}vy (1937) and the second due to
Jessen and Wintner (1935), see e.g. \cite{Lukacs}, Thm. 3.7.6 and 3.7.7
respectively. For the third property, we refer to \cite{BGMS}.

\begin{itemize}
\item Each $\mathfrak{I}$-distribution $\mathcal{N}$ in its Lebesgue
decomposition contains no discrete component.

\item Assume that the $\{\eps_{k}\}$ from (\ref{Class I}) are discrete. 
Then the measure $\mathcal{N}$ is either singular or it is absolutely
continuous w.r.t. Lebesgue measure.

\item Assume that the common characteristic function $\phi$ of the i.i.d. 
$\{\eps_{k}\}$ tends to zero at infinity and that the coefficients satisfy 
$b_{k}\geq C\exp (-Dk)$, for some $C,D>0$ and all $k$. Then $\mathcal{N}$
admits a $C^{\infty }$-density w.r.t. Lebesgue measure.
\end{itemize}

Various examples of singular random variables $\eps$ having characteristic functions $\phi$ 
which tend to zero at infinity are given in Section 3 of \cite{Lukacs} and also in Sections 6 and 7 
of \cite{GrahamMcGehee}. 
Here is an example due to Kerschner (1936): Suppose $a$ is a rational number such that 
$0<a<1/2$, and $a\neq 1/n$ for any integer $n\geq 3$. Then
\begin{equation*}
\phi(x)=\prod_{k=1}^{\infty }\cos \left(a^{k}x\right)
\end{equation*}
is the characteristic function of a singular symmetric $\mathfrak{J}$-distribution satisfying
\begin{equation}
\left\vert \phi(x)\right\vert \leq \frac{c}{(\log |x|)^{\gamma}}\text{ at }\infty  
\label{phi at infty}
\end{equation}
for some $\gamma ,c>0$.

\section{The Poisson Convergence}
\label{sec:poisson convergence}
\setcounter{equation}{0}

Fix a horocycle $H$. The eigenvalues $\lambda(B)$ can be represented by the empirical process
\begin{equation*}
M_{O}=\sum_{B\in \mathcal{B}_{HO}}\delta_{\lambda(B)}.
\end{equation*}
The \emph{intensity measure} $\mu_{O}(I)$ gives the expected number of 
$\lambda(B)$, $B\in\mathcal{B}_{HO}$, which fall in the interval $I$.
It is computed as
\begin{equation*}
\mu _{O}(I)=\E[M_{O}(I)]
=\left\vert \mathcal{B}_{HO}\right\vert\P\left\{\lambda(B)\in I \right\}.
\end{equation*}
Recall that the right-hand side of the above equality does not depend on 
$B\in \mathcal{B}_{HO}$. We fix numbers $c,\tau _{0}>0$ and consider a
small interval 
\begin{equation}
I_{O}=\Big\{\tau:\left\vert \tau -\tau _{0}\right\vert 
\leq \frac{c}{2\left\vert \mathcal{B}_{HO}\right\vert }\Big\}.  
\label{Interval}
\end{equation}
Assuming that the density of states $\mathfrak{m}(\tau):=\mathfrak{m}_{H}(\tau)$
(as defined in \eqref{def=ids}) exists and is continuous at $\tau_{0}$, 
and that $\mathfrak{m}\left(\tau _{0}\right)>0$, we obtain
\begin{equation}
\lim_{O\rightarrow\varpi}\mu_{O}(I_{O})=c\mathfrak{m}\left( \tau _{0}\right)
=:\lambda _{c}>0.  
\label{intensivity}
\end{equation}
If the $\lambda(B)$, $B\in \mathcal{B}_{HO}$, were i.i.d., then 
\eqref{intensivity} would yield the classical
convergence of the random variables $M_{O}(I_{O})$ to the
Poisson distribution $\Poi(\lambda)$ with intensity 
$\lambda:=\lambda _{c}$. More precisely, in the i.i.d.\ case we would then have 
(see \cite{LeadbetterLingremRootzen}, \cite{Diaconis}, \cite{AGG}) 
\begin{equation*}
\left\Vert \mathscr{L}(M_{O}(I_{O}))-\Poi(\lambda_{O})\right\Vert _{TV}
\leq \frac{\min (\lambda_{O},\lambda_{O}^{2})}{2\left\vert\mathcal{B}_{HO}\right\vert}
\end{equation*}
where $\lambda_{O}=\mu_{O}(I_{O})\rightarrow\lambda$ as $O\rightarrow\varpi$.
However, in our case the $\lambda(B)$, $B\in\mathcal{B}_{HO}$, are 
\emph{dependent} random variables, whence the classical theory does not apply
directly and needs some justifications and complements. 

In the course of our study we assume that 
\begin{equation}
1/\kappa \leq C(B)m(B)^{\delta}\leq \kappa,
\end{equation}
equivalently,
\begin{equation}
1/2\kappa \leq \lambda(B)m(B)^{\delta}\leq 2\kappa,
\label{delta condition}
\end{equation}
holds for all $B\in\mathcal{B}$ and some $\delta >1,\kappa >0$.
Assume also that the common law $\mathscr{L}(\eps)$ of i.i.d. $\eps(B)$ admits a bounded density. 
Then it was shown in \cite{BGMS} that
\begin{equation}
\mathscr{L}(M_{O}(I_{O}))\rightarrow \Poi(\lambda )\text{ weakly as \ } O\rightarrow \varpi. 
\label{eq:poisconv}
\end{equation}
Examples provided in \cite{BGMS} show that the assumption that $\mathscr{L}(\eps)$ admits 
a bounded density can not be dropped entirely. Whether the assumption that 
\eqref{delta condition} holds with $\delta >1$ can be relaxed is very much open at present writing.

We aim to estimate the rate of convergence of $\mathscr{L}\left(M_{O}(I_{O})\right) \to \Poi(\lambda)$ 
in total variation. To this end, we apply an approximation result from \cite{AGG}. 
Recall that for any $B\in H,$ the eigenvalue $\lambda(B)$ can be written in the form
\begin{equation}
\lambda(B)=\lambda_{H}\left(1+U(B)\right)
\label{lambda-U}
\end{equation}
where
\begin{equation}
U(B)=\sum_{B\subseteq B_{k}}\alpha_{k}\eps(B_{k})  
\label{U_H representation}
\end{equation}
and $\alpha_{k}=C(B_{k})/\lambda _{H}$. The common distribution function $\mathcal{N}_H$ 
of the family $\{U(B)\}_{B\in H}$ is absolutely continuous and its density $\mathfrak{n}_{H}$ 
relates to the integrated density of states $\mathfrak{m}_{H}$ by the equation 
\begin{equation*}
\mathfrak{n}_{H}(t)=\lambda _{H}\mathfrak{m}_{H}(\lambda_{H}t+\lambda _{H}).
\end{equation*}
In particular, $\mathfrak{n}_{H}$ is supported by some interval $[-\epsilon,\epsilon ]\subset (-1,1)$, 
is continuous, and is strictly positive at $t_{0}=\tau_{0}/\lambda_{H}-1$. 
Let $N_{O}$ be the empirical process defined by the family $\left\{U(B)\right\}$, 
$B\in \mathcal{B}_{HO}$, and choose the interval $J_{O}$ as 
\begin{equation*}
J_{O}=\Big\{t:\left\vert t-t_{0}\right\vert 
\leq \frac{d}{2\left\vert \mathcal{B}_{HO}\right\vert }\Big\} ,\text{ }
d=c/\lambda_{H}.
\end{equation*}
Then equation \eqref{lambda-U} yields
\begin{equation*}
\P\left\{M_{O}(I_{O})=k\right\}=\P\left\{N_{O}(J_{O})=k\right\}.
\end{equation*}
Thus in what follows, we can restrict our attention to the family $\{U(B)\}_{B\in H}$. 
Since we fix a horocycle $H$ we can further simplify our task by replacing $X$ with the 
discrete Abelian group 
\begin{equation*}
G=\bigoplus_{k\geq 1}\Z\left(n_{k}\right),\text{ } n_{k}\in\N_{\geq 2}, 
\end{equation*}
equipped with its canonical ultrametric structure defined by the family 
$\{G_{\ell}\}_{\ell\geq 0}$ of its finite subgroups 
\begin{equation*}
G_{0}=\{e\},\text{ }G_{\ell}=\prod_{1\leq k\leq \ell}\Z\left(n_{k}\right) .
\end{equation*}
Denote the order of $G_{\ell}$ by $\pi_{\ell}$ so that $\pi_{\ell}=n_{0}n_{1}\cdots n_{\ell}$ 
with the convention that $n_{0}=1$.
We have that $H=H_{0}$ is the set of all singletons, $H_{1}$ is the set of all
ultrametric balls of the form $g+G_{1}$, and so forth. 

When $B=\{g\}$ we write $U(B)=U_{g}(\omega)$. 
For the increasing sequence of balls $B_{k}\supseteq B$, 
we set $\eps(B_{k})=\eps(g_{k})$. 
We assume henceforth that the random variables $\{\eps (g_{k})\}_{g\in G,k\in\N_{0}}$ 
are independent and satisfy 
\begin{itemize}
\item $\left|\eps(g_{k})\right|<1$ for all $g \in G$, $k \geq 0$.

\item $\{\eps(g_{k})\}_{g\in G}$ are identically distributed for each $k\geq 0$.

\item $\eps(e_{0})$ is absolutely continuous with bounded density $\eta_{\eps}$.
\end{itemize}
In particular, there is no symmetry assumption, $\eps(g_k)$ and $\eps(g_\ell)$ do not have to have 
the same distribution for $k\neq \ell$, and $\eps(g_k)$ does not have to be absolutely continuous 
(or have bounded density) for $k>0$. Though these relaxations are not all applicable to the setting 
of perturbed hierarchical Laplacians as previously defined, it is of interest to prove a quantitative law of 
rare events with the same basic dependence structure in as much generality as possible.

Define the stationary process $\{U_{g}\}_{g\in G}$ by
\begin{equation*}
U_{g}=\sum_{k=0}^{\infty }\alpha_{k}\eps_{g_{k}}.
\end{equation*}
In what follows, we need only assume that the sequence $\{\alpha_{k}\}_{k=0}^{\infty}$ satisfies 
\begin{equation*}
\sum_{k=\ell}^{\infty}\alpha_{k}\leq K\pi_{\ell}^{-(1+\gamma)}
\end{equation*}
for some constants $K,\gamma>0$. Here we are thinking of $\gamma=\delta - 1$, 
a more convenient notation for subsequent computations.
Note that absolute continuity of $\eps(g_{0})$ ensures that $U_{g}$ is absolutely continuous. 
Moreover, writing $U_{g}=Y_{g}^{0}+T_{g}^{0}$ where $Y_{g}^{0}=\alpha_{0}\eps(g_{0})$ 
has density $f_{0}(x)=\frac{1}{\alpha_{0}}\eta_{\eps}\left(\frac{x}{\alpha_{0}}\right)$ and 
$T_{g}^{0}=\sum_{k=1}^{\infty}\alpha_{k}\eps(g_{k})$ has distribution $\nu_{0}$, we see that the 
density of $U_{g}$ satisfies
\begin{align} \label{eq:convolution}
\eta(x)=\int f_{0}(x-y)d\nu_{0}(y)\leq \norm{f_{0}}\int d\nu_{0}(y) =\frac{1}{\alpha_{0}}\norm{\eta_{\eps}}
\end{align}
for all $x\in\R$. (We write $\eta$ rather than $\mathfrak{n}$ to emphasize that we are working with 
slightly weaker assumptions than are used in the proof of \eqref{eq:poisconv} from \cite{BGMS}.)  

Let $t_{0}$ be a Lebesgue point of $\eta$ such that $\eta\left(t_{0}\right)>0$ and let $c$ be some 
fixed positive number. 
Set $I_{\ell}=\left[t_{0}-\frac{c}{2\pi_{\ell}},t_{0}+\frac{c}{2\pi_{\ell}}\right]$, 
and define 
\[
X_{g}^{\ell}=1\{U_{g}\in I_{\ell}\}, \ W_{\ell}=\sum_{g\in G_{\ell}}X_{g}^{\ell}.
\]
We are interested in providing error terms for the Poisson approximation of $W_{\ell}$, the analogue of 
$N_{O}(J_{O})$.

For each $\ell\geq 0$, the $X_{g}^{\ell}$'s are identically distributed, so if 
$p_{\ell}=\P\left\{U_{e}\in I_{\ell}\right\}$, then $\E\left[W_{\ell}\right]=\pi_{\ell}p_{\ell}$. 
Writing $\lambda(\ell)$ for this expectation, it follows from the Lebesgue differentiation theorem that
\[
\lim_{\ell\rightarrow\infty}\lambda(\ell)
=c\lim_{\ell\rightarrow\infty}\frac{\pi_{\ell}}{c}\int_{t_{0}-\frac{c}{2\pi_{\ell}}}^{t_{0}
+\frac{c}{2\pi_{\ell}}}\eta(x)dx=c\eta(t_{0})=:\lambda.
\]
We now state our main result, which proves a quantitative version of the convergence in \eqref{eq:poisconv}.

\begin{theorem}
\label{thm: main}
Keeping the above assumptions, set $m=\inf_{j\geq 1}n_{j}$, $M=\sup_{j\geq 1}n_{j}$,
\[
C=\frac{1-e^{-\lambda(\ell)}}{\lambda(\ell)} \left( \lambda(\ell)^{2} 
+ \frac{c^{2}}{\alpha_0^2} \norm{\eta_{\eps}}^{2} \right)
+\frac{16K}{\alpha_0}\left(1\wedge\lambda(\ell)^{-\frac{1}{2}}\right)\norm{\eta_{\eps}},
\]
and let $M_{\ell}$ be the running maximum of the subsequence of 
$\{n_{j}\}_{j=1}^{\infty}$ defined in \eqref{eq:index_choice}.
\begin{enumerate}[(a)]
\item \label{item:finite_sup_tv} If $M < \infty$, then 
\begin{align*}
\tv{\mathscr{L}\left(W_{\ell}\right)-\Poi\left(\lambda(\ell)\right)}
\leq Cm^{-\ell \frac{\gamma}{\log_{m}(M)+1+\gamma}}.
\end{align*}

\item \label{item:infinite_sup_tv} If $M = \infty$, then
\begin{align*}
\tv{\mathscr{L}\left(W_{\ell}\right)-\Poi\left(\lambda(\ell)\right)}\leq C M_{\ell-1}^{-\gamma/3}.
\end{align*}
\end{enumerate}
\end{theorem}

Note that $\displaystyle{\frac{1-e^{-x}}{x} \leq 1\wedge x^{-1}}$ and
\[
\lambda(\ell)=\pi_{\ell}\int_{I_{\ell}}\eta(x)dx \leq \pi_{\ell}|I_{\ell}|\norm{\eta}
\leq\frac{c\norm{\eta_{\eps}}}{\alpha_{0}},
\]
so we have
\begin{equation}
\label{C bound}
C \leq \lambda(\ell)+\frac{c^{2}\norm{\eta_{\eps}}^{2}}{\alpha_{0}^{2}}
+\frac{16K\norm{\eta_{\eps}}}{\alpha_{0}}
\leq \frac{\norm{\eta_{\eps}}}{\alpha_{0}}\left(c+\frac{c^{2}\norm{\eta_{\eps}}}{\alpha_{0}}
+16K\right),
\end{equation}
which is independent of both $t_0$ and $\ell$.

With Theorem~\ref{thm: main} in hand, the rate of convergence to the $\Poi(\lambda)$ distribution 
can be estimated as 
\begin{multline*}
\tv{\mathscr{L}\left(W_{\ell}\right)-\Poi(\lambda)} \leq \tv{\mathscr{L}\left(W_{\ell}\right)-\Poi\left(\lambda(\ell)\right)}
+\tv{\Poi\left(\lambda(\ell)\right)-\Poi(\lambda)}\\
 =\tv{\mathscr{L}\left(W_{\ell}\right)-\Poi\left(\lambda(\ell)\right)}
+\frac{1}{2}\sum_{k=0}^{\infty}\frac{1}{k!}\left|\lambda(\ell)^{k}e^{-\lambda(\ell)}-\lambda^{k}e^{-\lambda}\right|,
\end{multline*}
which goes to zero since $\lambda(\ell)\rightarrow\lambda$. See \cite{AdJod} for more concise 
bounds on the second term in the final expression. 

The next section is devoted to the proof of Theorem~\ref{thm: main}, but first we consider a concrete 
application of this result. Let $\Q_{p}$ be the field of $p$-adic numbers and let $\mathfrak{D}^{\alpha}$ 
be the fractional derivative operator of order $\alpha$ defined in \eqref{frac-deriv}; 
see also \eqref{C-B-fractur}.
Then $\mathfrak{D}^{\alpha}$ is a hierarchical Laplacian corresponding to the function 
$C(B)=(p^{\alpha}-1)m(B)^{-\alpha}$, so its eigenvalues are $\lambda(B)=p^{\alpha}m(B)^{-\alpha}$. 
Every $p$-adic ball in the horocycle $H_k$ is of the form $B=a+p^{k}\Z_{p}$ (where $\Z_{p}$ is the 
compact group of $p$-adic integers) and thus has Haar measure $m(B)=p^{-k}$.
Therefore every $B\in H_{k}$ gives rise to the eigenvalue $\lambda_{k}=p^{\alpha(k+1)}$. 

Without loss of generality, we fix the horocycle $H=H_{-1}$ corresponding to the eigenvalue 
$\lambda_{-1}=1$. We consider a perturbation of $\mathfrak{D}^{\alpha}$ by i.i.d. 
$\eps(B)\sim\text{Unif}(-1,1)$.
Since $H$ is fixed, we can replace $\Q_{p}$ with $G=\bigoplus_{k\geq1}\Z(n_{k})$ where $n_{k}=p$ 
for all $k$. One readily checks that \eqref{delta condition} is satisfied with $\delta = \alpha$ and 
$\alpha>1$. 
Setting $\alpha_{k}=C(B_{-1-k})/\lambda_{-1}=(1-p^{-\alpha})p^{-\alpha k}$, we have
\[
\sum_{k=\ell}^{\infty}\alpha_{k}
=(1-p^{-\alpha})\sum_{k=\ell}^{\infty}p^{-\alpha k}
=p^{-\alpha\ell}=K\pi_{\ell}^{-(1+\gamma)}
\]
with $\gamma=\alpha-1$, $K=1$. 

For the sake of notational simplicity, we fix $t_{0}=0$ and $c=\pi$. We have that 
$U=U_{e}=\sum_{k=0}^{\infty}\alpha_{k}\eps_{k}$ with $\{\eps_{k}\}_{k=0}^{\infty}$
i.i.d. $\text{Unif}(-1,1)$, so $U$ has characteristic function
$\varphi(t)=\prod_{k=0}^{\infty}\frac{\sin(\alpha_{k}t)}{\alpha_{k}t}$ (where we take 
$\frac{\sin(0)}{0}=1$).
Fourier inversion \cite{Dur} gives 
\[
\lambda=\pi\eta(0)
=\frac{1}{2}\int_{-\infty}^{\infty} \prod_{k=0}^{\infty}\frac{\sin(\alpha_{k}s)}{\alpha_{k}s}ds
\]
and 
\begin{align*}
\lambda(\ell) & =p^{\ell}\P\left(-\frac{\pi}{2p^{\ell}}\leq U\leq\frac{\pi}{2p^{\ell}}\right)\\
 & =\frac{p^{\ell}}{2\pi}\int_{-\infty}^{\infty} \frac{\exp\left(\frac{\pi is}{2p^{\ell}}\right)
-\exp\left(-\frac{\pi is}{2p^{\ell}}\right)}{is}\varphi(s)ds\\
 & =\frac{1}{2}\int_{-\infty}^{\infty} \frac{2p^{\ell}}{\pi s} \sin\left(\frac{\pi s}{2p^{\ell}}\right)
\prod_{k=0}^{\infty}\frac{\sin(\alpha_{k}s)}{\alpha_{k}s}ds.
\end{align*}

Since $\left\Vert \eta_{\eps}\right\Vert _{\infty}=\frac{1}{2}$,
$m=M=p$, and $\alpha_{0}=1-p^{-\alpha}$, Theorem \ref{thm: main} shows that 
\[
\left\Vert \mathscr{L}(W_{\ell})-\Poi(\lambda(\ell))\right\Vert _{TV}\leq C p^{-\frac{\alpha-1}{\alpha+1}\ell}
\]
where, by Equation \eqref{C bound},
\[
C\leq\frac{1}{2(1-p^{-\alpha})}\left(\pi+\frac{\pi^{2}}{2(1-p^{-\alpha})}+16\right)<30.
\]
In particular, the bound on the error is of order 
\[
p^{-\frac{\alpha-1}{\alpha+1}\ell}=\left(\frac{1}{|B_{-\ell}|}\right)^{\frac{\alpha-1}{\alpha+1}}, 
\]
whereas one would expect a $O\left(\frac{1}{|B_{-\ell}|}\right)$ bound if the $\lambda(B)$ terms 
were independent.

\section{Convergence Rates}
\setcounter{equation}{0}

To prove Theorem~\ref{thm: main}, we appeal to the following result from \cite{AGG}.
\begin{theorem}
\label{thm:AGG}
Suppose that $\{X_{\alpha}\}_{\alpha\in J}$ are Bernoulli random variables with success probabilities 
$p_{\alpha}=\P\left\{X_{\alpha}=1\right\}\in(0,1)$.
For $\alpha\neq\beta$, write $p_{\alpha\beta}=\E\left[X_{\alpha}X_{\beta}\right]$.
Set $W=\sum_{\alpha\in J}X_{\alpha}$ and assume that $\mu=\E\left[W\right]<\infty$.
For each $\alpha\in J$, let $B_{\alpha}\subseteq J$ be a subset containing $\alpha$. 
Define $V_{\alpha}=\sum_{\beta\in J\setminus B_{\alpha}}X_{\beta}$,
\begin{align*}
b_{1} & =\sum_{\alpha\in J}\sum_{\beta\in B_{\alpha}}p_{\alpha}p_{\beta},\\
b_{2} & =\sum_{\alpha\in J}\sum_{\beta\in B_{\alpha}\setminus\{\alpha\}}p_{\alpha\beta},\\
b_{3} & =\sup_{f\in\mathcal{F}}\sum_{\alpha\in J}\E\left[(X_{\alpha}-p_{\alpha})f(V_{\alpha}+1)\right],
\end{align*}
where $\mathcal{F}=\left\{f:\N_{0} \to \R\textrm{ such that }\norm{f}\leq1\wedge\mu^{-\frac{1}{2}}\right\}$.
Then
\[
\tv{\mathscr{L}(W)-\Poi(\mu)}\leq\frac{1-e^{-\mu}}{\mu}(b_{1}+b_{2})+b_{3}.
\]
\end{theorem}

For the problem considered here, we have $J=G_{\ell}$, $X_{\alpha}=X_{g}^{\ell}$, $\mu=\lambda(\ell)$. 
We take our \emph{dependency neighborhoods} to be the closed $k$-balls
\begin{align} \label{eq:dep nbhds}
B_{g}=g+G_{k}
\end{align}
for some $k = k(\ell)\in [0,\ell)$ to be determined. The following lemma provides bounds on 
$b_1, b_2$ and $b_3$ in terms of the size of the dependency neighborhoods $B_{g}$.

\begin{lemma} 
\label{lem:bounds}
Let $b_1, b_2$ and $b_3$ be the constants introduced in Theorem~\ref{thm:AGG} where the 
other terms are interpreted as in the preceding paragraph. Then
\begin{align*}
b_{1} & = \lambda(\ell)^{2} \pi_{k}\pi_{\ell}^{-1},\\
b_{2} & \leq \frac{c^2}{\alpha_{0}^{2}}\norm{\eta_{\eps}}^2 \pi_{k}\pi_{\ell}^{-1} ,\\
b_{3} & \leq \frac{16K}{\alpha_0}\left(1\wedge\lambda(\ell)^{-\frac{1}{2}}\right)\norm{\eta_{\eps}} \pi_{\ell}\pi_{k+1}^{-(1+\gamma)}.
\end{align*}
\end{lemma}

The proof of Lemma~\ref{lem:bounds} is postponed until Subsection~\ref{sec:bounds_proof}. 
For the moment, note that the dependency neighborhoods \eqref{eq:dep nbhds} must be chosen with care. 
If $k$ is too large, then the $b_1$ and $b_2$ estimates do not  go to zero as $\ell\rightarrow\infty$. 
If it is too small, then the bound on $b_3$ does not go to zero. Also, it is worth observing that 
in most known applications of Theorem~\ref{thm:AGG}, the dependency neighborhoods can be chosen 
so that $b_{3}=0$. When there is long-range dependence, it is typically more efficient to get 
error bounds using size-bias couplings \cite{AGG1,BHJ,Ross}. However, in the case at hand, 
it seems that Theorem~\ref{thm:AGG} is the most straightforward approach, the difficulty of bounding $b_{3}$ notwithstanding. 

The next lemma guarantees that the dependency neighborhoods can always be chosen to make the bounds in 
Lemma~\ref{lem:bounds} vanish as $\ell \to \infty$. To state it, we introduce some notation. Recall the sequence 
$\{n_j\}_{j = 1}^{\infty}\subseteq \N_{\geq 2}$ defining $G$ and the choice of $n_0=1$. Set 
\begin{align} \label{eq:def_M}
m = \inf_{j \geq 1} n_j, \quad M = \sup_{j \geq 1} n_j.
\end{align}
When $M = \infty$, define $\{n_{j_i}\}_{i=0}^{\infty} \subseteq \{n_j\}_{j=0}^{\infty}$ to be the subsequence where
\begin{align} \label{eq:index_choice}
j_0=0 \eqand j_{i+1} = \min\{j>j_i :n_j > n_{j_i} \text{ and } n_j > n_{j_i}^{\gamma/3}\},
\end{align}
and observe that since $n_{j_0}=n_0=1$, we have $j_1 = 1$. For $\ell \geq 0$ we define 
\begin{align} \label{eq:def_M_l}
M_{\ell} = \sup \{n_{j_i} : j_i \leq \ell \}.
\end{align}
 We now have the following result.
\begin{lemma}
\label{lem:dep_nhoods}
Let $m$, $M$, and $M_{\ell}$ be defined as in \eqref{eq:def_M} and \eqref{eq:def_M_l}, respectively. 
Then for every $\ell\geq 1$, we can always choose $k = k(\ell)<\ell$ such that
\begin{enumerate}[(a)]
\item \label{item:finite_sup} If $M < \infty$, then 
\begin{align*}
\pi_k\pi_{\ell}^{-1},\pi_{\ell} \pi_{k+1}^{-(1+\gamma)} 
\leq m^{-\ell\frac{\gamma}{\log_{m}(M)+1+\gamma}}.
\end{align*}

\item \label{item:infinite_sup} If $M = \infty$, then
\begin{align*}
\pi_k\pi_{\ell}^{-1},\pi_{\ell} \pi_{k+1}^{-(1+\gamma)} \leq M_{\ell-1}^{-\gamma/3}.
\end{align*}
\end{enumerate}
\end{lemma}

Observe that Theorem~\ref{thm: main} is now a direct consequence of Theorem~\ref{thm:AGG}, 
together with Lemmas \ref{lem:bounds} and \ref{lem:dep_nhoods}. 
The rest of the paper is devoted to proving the two lemmas.

\begin{proof}[Proof of Lemma~\ref{lem:dep_nhoods}] 
Suppose that $M < \infty$. Then
\begin{align*}
\pi_{\ell}\pi_{k+1}^{-(1+\gamma)} 
=  \frac{n_{k+2} \times \ldots \times n_{\ell}}{(n_1 \times \ldots \times n_{k+1})^\gamma} 
\leq  M^{\ell-(k+1)}m^{-\gamma (k+1)} =  m^{\ell \log_{m}(M) -(k+1)( \log_{m}(M)+ \gamma)}
\end{align*}
and 
\begin{align*}
\pi_k\pi_{\ell}^{-1} =  \frac{1}{n_{k+1} \times \ldots \times n_{\ell}} \leq  m^{k-\ell}.
\end{align*}
Choosing
\begin{align*}
k = \left \lfloor\frac{\ell\left(\log_{m}(M)+1\right)}{\log_{m}(M)+1+\gamma}\right\rfloor
\end{align*}
and using $\left\lfloor x\right\rfloor \leq x < \left\lfloor x\right\rfloor +1$ completes the proof of part \eqref{item:finite_sup}. 

Part \eqref{item:infinite_sup} is much more involved. To prove it, we specify a rule for picking $k = k(\ell)$ 
for every $\ell \geq 1$. Suppose that $M = \infty$ and recall the subsequence $\{n_{j_i}\}_{i=0}^{\infty} $ introduced in \eqref{eq:index_choice}. 
When $\ell = 1$, we choose $k(1) = 0$, and since $M_{0} = 1$, we see that 
\begin{align*}
\pi_k\pi_{\ell}^{-1} = \pi_0\pi_{1}^{-1} = \frac{1}{n_1} \leq M_{0}^{-\gamma/3}  \eqand \pi_{\ell}\pi_{k+1}^{-(1+\gamma)} 
= \pi_{1} \pi_{1}^{-(1+\gamma)} = \frac{1}{n_1^{\gamma}} \leq M_{0}^{-\gamma/3}.
\end{align*}
Fix $i\geq 1$. We now outline a procedure to choose $k$ when $\ell \in  (j_i, j_{i+1}]$  such that 
\begin{align} \label{eq:size_ratio_ub}
\pi_k\pi_{\ell}^{-1},\pi_{\ell} \pi_{k+1}^{-(1+\gamma)} \leq n_{j_i}^{-\gamma/3}.
\end{align}
Since $j_1 = 1$, this procedure will specify $k(\ell)$ for all $\ell \geq 2$, and the result will follow since 
$\ell \in  (j_i, j_{i+1}]$ if and only if $M_{\ell-1}=n_{j_i}$.

We first define
\begin{align*}
B = \{b \in (j_i, j_{i+1}] : n_{b} \geq n_{j_i}^{\gamma/3}\} \eqand S = (j_i, j_{i+1}]\setminus B.
\end{align*}
 The sets $B$ and $S$ should be interpreted as sets of ``big'' and ``small'' indices, respectively. 
For any integer $r\in [{j_{i}}-1, {j_{i+1}}-2]$, define the integers $u(r)$ and $v(r)$ by
\begin{align*}
u(r) = \min \left\{u \in [r+2, {j_{i+1}}]: u \in B \text{ or }   \prod \limits_{ j = r+2}^{u} n_{j} \geq n_{j_{i}}^{2\gamma / 3}\right\}
\end{align*}
and 
\begin{align*}
v(r) + 1 =
\begin{cases}
\min \Big\{v \in [r+2, {j_{i+1}}]:\prod \limits_{ j = r+2}^{v} n_{j} \geq n_{j_{i}}^{\gamma / 3}\Big\}, \quad u(r) \in S, \\
u(r), \quad u(r) \in B.  
\end{cases}
\end{align*}
We see that these quantities are always well defined, because $j_{i+1} \in B$ by \eqref{eq:index_choice}.

One may interpret $u(r)$ and $v(r)$ as follows: Suppose $\ell_0 \in (j_i, j_{i+1})$ and $k_0 = k(\ell_0)< \ell_0$ satisfies \eqref{eq:size_ratio_ub}. 
Then $u(k_0)$ tells us how far we can increase $\ell$ past $\ell_0$ while keeping $k(\ell)$ fixed at $k_0$ and not violating \eqref{eq:size_ratio_ub}. 
For example, suppose we keep $k(u(k_0)) = k_0$ once $\ell$ reaches $u(k_0)$, then we are no longer guaranteed that 
$\pi_{u(k_0)} \pi_{k_0+1}^{-(1+\gamma)} \leq n_{j_i}^{-\gamma/3}$. We see that $k(u(k_0))$ must be increased past $k_0$, 
but it cannot be increased too much, because we must maintain $\pi_{k(u(k_0))}\pi_{u(k_0)}^{-1} \leq n_{j_i}^{-\gamma/3}$. 
Hence, we use $v(u(k_0))$ to tell us exactly how far to increase $k(u(k_0))$. We then repeat this process inductively.

We now formalize the discussion above. For any $\ell \in [j_{i},j_{i+1}]$, we set
\begin{align*}
k(\ell) = 
\begin{cases}
k(w_{m-1}), \quad w_{m-1} \leq \ell < w_m, \\
v(k(w_{m-1})), \quad \ell = w_m,
\end{cases} 
\end{align*}
where 
\begin{align*}
w_0 = j_i, \quad k(w_0) = j_i - 1 \eqand w_{m} = u(k(w_{m-1}))
\end{align*}
for all $m \geq 1$ such that $w_m \leq j_{i+1}$. 

We now verify that choosing $k(\ell)$ according to this procedure satisfies \eqref{eq:size_ratio_ub} for all $\ell \in (j_{i},j_{i+1}]$.
Observe that when $u(r) \in S$, then
\begin{align}
\prod \limits_{ j = r +2}^{u(r)-1} &n_{j}  < n_{j_{i}}^{2\gamma / 3}, \quad \prod \limits_{ j = r +2}^{u(r)} n_{j}  < n_{j_{i}}^{\gamma},\label{eq:u_ineq}
\end{align}
and
\begin{align}
\prod \limits_{ j = r+2}^{v(r)} n_{j} <  n_{j_{i}}^{\gamma / 3}, \label{eq:v_ineq}
\end{align}
with the convention that the empty product equals $1$. For $m \geq 1$, if $\ell = w_m \in B$, 
then $k(\ell) = \ell - 1$ and
\begin{align*}
\pi_{\ell} \pi_{k+1}^{-(1+\gamma)}= \pi_{\ell}^{-\gamma} \leq n_{j_i}^{-\gamma/3} \eqand \pi_k \pi_{\ell}^{-1}   
= n_{w_m}^{-1} \leq n_{j_i}^{-\gamma/3}.
\end{align*}

Next, suppose that $\ell = w_m \in S$ for some $m \geq 1$ and  $w_m < j_{i+1}$. 
Then using \eqref{eq:u_ineq} together with the definition $v(r)$ and the fact that $j_i \leq k(w_{m-1})+1$, we see that
\begin{multline*}
\pi_{\ell} \pi_{k+1}^{-(1+\gamma)} = \pi_{w_m} \pi_{k(w_{m})+1}^{-(1+\gamma)} \\
= \left(\prod \limits_{ j = 1}^{k(w_{m-1})+1} n_{j}^{-\gamma}\right) \left(\prod \limits_{ j = k(w_{m-1})+2}^{u(k(w_{m-1}))} n_{j}\right)
\left(\prod \limits_{ j = k(w_{m-1})+2}^{v(k(w_{m-1}))+1} n_{j}^{-(1+\gamma)}\right) \\
<\ n_{j_i}^{-\gamma} n_{j_i}^{\gamma} n_{j_i}^{-\gamma/3} = n_{j_i}^{-\gamma/3}.
\end{multline*}
Furthermore, \eqref{eq:v_ineq} implies
\begin{multline*}
\pi_{k} \pi_{\ell}^{-1}   
=\pi_{k(w_{m})} \pi_{w_m}^{-1}\\
= \left(\prod \limits_{ j = k(w_{m-1})+2}^{v(k(w_{m-1}))} n_{j}\right) 
\left(\prod \limits_{ j = k(w_{m-1})+2}^{u(k(w_{m-1}))} n_{j}^{-1}\right)
 < n_{j_i}^{\gamma/3} n_{j_i}^{-2\gamma/3} = n_{j_i}^{-\gamma/3}.
\end{multline*}

Lastly, suppose $\ell \in (w_{m-1},w_m) \cap S$ for some $m \geq 1$ and  $w_m < j_{i+1}$. 
Then $k(\ell) = k(w_{m-1})$, and by \eqref{eq:u_ineq} we have
\begin{multline*}
\pi_{\ell} \pi_{k+1}^{-(1+\gamma)} \leq \pi_{w_m - 1} \pi_{k(w_{m-1})+1}^{-(1+\gamma)} \\
= \ \left(\prod \limits_{ j = 1}^{k(w_{m-1})+1} n_{j}^{-\gamma}\right) 
\left(\prod \limits_{ j = k(w_{m-1})+2}^{u(k(w_{m-1}))-1} n_{j}\right)
<\ n_{j_i}^{-\gamma} n_{j_i}^{2\gamma/3} = n_{j_i}^{-\gamma/3},
\end{multline*}
and
\begin{align*}
&\pi_{k} \pi_{\ell}^{-1}   
 \leq \pi_{k(w_{m-1})} \pi_{w_{m-1}}^{-1} \leq n_{j_i}^{-\gamma/3}.
\end{align*}
Hence, our procedure of choosing $k(\ell)$ satisfies \eqref{eq:size_ratio_ub} for all $\ell \in (j_i,j_{i+1}]$, and the proof is complete.
\end{proof}

\subsection*{Proof of Lemma~\ref{lem:bounds}} 
\label{sec:bounds_proof}
$ $

We first note that $\E\left[X_{g}^{\ell}\right]=p_{\ell}$ for all $g\in G_{\ell}$, so
\begin{align*}
b_{1}=\sum_{g\in G_{\ell}}\sum_{h\in B_{g}}p_{\ell}^{2}=\pi_{\ell}\pi_{k}p_{\ell}^{2}=\pi_{k}\pi_{\ell}^{-1}\lambda(\ell)^{2}
\end{align*}
as claimed.

We now move on to bound $b_{2} =\sum_{g\in G_{\ell}}\sum_{h\in B_{g}\setminus\{g\}}p_{gh}$. For $g\in G_{\ell}$, define 
\begin{align} \label{eq:def_YT}
Y_{g}^{j}=\sum_{i=0}^{j}\alpha_{i}\eps(g_{i}) \eqand T_{g}^{j}=\sum_{i=j+1}^{\infty}\alpha_{i}\eps(g_{i}), \quad j \geq 0.
\end{align}
For any $g\in G_{\ell}$ and $h\in B_{g}\setminus\{g\}$, there is a minimal $0\leq j\leq k-1$ such that 
\begin{align*}
U_{g}=Y_{g}^{j}+T_{g}^{j} \eqand U_{h}=Y_{h}^{j}+T_{g}^{j}.
\end{align*} 
Proposition~\ref{cond_exp_fcn} implies that 
\begin{align*}
p_{gh} & =\E\left[X_{g}^{\ell}X_{h}^{\ell}\right]=\P\left\{Y_{g}^{j}+T_{g}^{j}\in I_{\ell}, Y_{h}^{j}+T_{g}^{j}\in I_{\ell}\right\}\\
 & =\E\left[\P\left\{Y_{g}^{j}+T_{g}^{j}\in I_{\ell}, Y_{h}^{j}+T_{g}^{j}\in I_{\ell}\left|T_{g}^{j}\right.\right\}\right]
=\E\left[\P\left\{Y_{g}^{j}\in I_{\ell} - T_{g}^{j}\left|T_{g}^{j}\right.\right\}^{2}\right],
\end{align*}
where the last equality used the fact that $Y_{g}^{j}$ and $Y_{h}^{j}$ are independent and equal in distribution. 

Arguing as in \eqref{eq:convolution} shows that $Y_{g}^{j}$ has density $f_{j}$ satisfying 
$\norm{f_{j}}\leq\frac{1}{\alpha_{0}}\norm{\eta_{\eps}}$. Thus, since $Y_{g}^{j}$ and $T_{g}^{j}$ are independent, we have
$\P\left\{Y_{g}^{j}\in I_{\ell} - T_{g}^{j}\left|T_{g}^{j}\right.\right\}=\varphi\left(T_{g}^{j}\right)$ where 
\[
\varphi(t)=\P\left\{Y_{g}^{j}\in I_{\ell}-t\right\}=\int_{I_{\ell}-t} f_{j}(x)dx \leq \norm{f_{j}}\abs{I_{\ell}-t}=\frac{1}{\alpha_{0}}\norm{\eta_{\eps}}c\pi_{\ell}^{-1}.
\]
It follows that
\begin{align*}
b_{2} =\sum_{g\in G_{\ell}}\sum_{h\in B_{g}\setminus\{g\}}p_{gh}
=&\ \pi_{\ell}\sum_{j=0}^{k-1}(n_{j+1}-1)\pi_{j}\E\left[\P\left\{Y_{g}^{j}\in I_{\ell} - T_{g}^{j}\left|T_{g}^{j}\right.\right\}^{2}\right] \\
\leq&\ \pi_{\ell}(\pi_{k}-1)\frac{1}{\alpha_0^2}\norm{\eta_{\eps}}^2 c^2 \pi_{\ell}^{-2} \leq \pi_{k}\frac{1}{\alpha_0^2}\norm{\eta_{\eps}}^2 c^2 \pi_{\ell}^{-1} 
\end{align*}
where the first inequality used $\sum_{j=0}^{k-1}\left(n_{j+1}-1\right)\pi_{j} = \pi_{k}-1$.
Observe that this bound is of the correct order of magnitude since Jensen's inequality shows that
\begin{align*}
\E\left[\P\left\{Y_{g}^{j}\in I_{\ell} - T_{g}^{j}\left|T_{g}^{j}\right.\right\}^{2}\right] 
\geq \E\left[\P\left\{Y_{g}^{j}\in I_{\ell} - T_{g}^{j}\left|T_{g}^{j}\right.\right\}\right]^{2} 
= p_{\ell}^2 = O(\pi_{\ell}^{-2}).
\end{align*}

We finish with the bound on $b_3$. Let $g^{(1)}, \ldots ,\ g^{(m)}$ be a transversal for $G_{k}$ in $G_{\ell}$ 
with $g^{(1)}=e$ and $m=\pi_{\ell}\pi_{k}^{-1}$. 
Then for any $g \in G_{\ell}$, the dependency neighborhood $B_g$ is one of 
$B^{(1)},\ldots,B^{(m)}$, where 
\begin{align*}
B^{(i)}:=g^{(i)}+G_{k}, \quad 1 \leq i \leq m.
\end{align*}
Let $\mathcal{G}^{(i)}$ be any sub-$\sigma$-field with respect to which 
$V^{(i)}:=V_{g^{(i)}}=\sum_{g\in G_{\ell}\setminus B^{(i)}}X_{g}^{\ell}$ is measurable. 
Then, noting that 
\begin{align*}
\left(V^{(i)},\sum_{g\in B^{(i)}}X_{g}^{\ell}\right)
\stackrel{d}{=} \left(V^{(j)},\sum_{g\in B^{(j)}}X_{g}^{\ell}\right)
\end{align*}
for all $1 \leq i,j \leq m$, we have
\begin{align*} 
\sum_{g\in G_{\ell}}\E\left[(X_{g}^{\ell}-p_{\ell})f(V_{g}+1)\right] 
 & =\sum_{i=1}^{m}\E\left[f(V^{(i)}+1)\sum_{g\in B^{(i)}}\left(X_{g}^{\ell}-p_{\ell}\right)\right] \notag \\
 & =\pi_{\ell}\pi_{k}^{-1}\E\left[f(V^{(1)}+1)\sum_{g\in G_{k}}\left(X_{g}^{\ell}-p_{\ell}\right)\right] \notag \\
 & =\pi_{\ell}\pi_{k}^{-1}\E\left[f(V^{(1)}+1)\E\left[\sum_{g\in G_{k}}\left(X_{g}^{\ell}-p_{\ell}\right)\left|\mathcal{G}^{(1)}\right.\right]\right] \notag \\
 & \leq\pi_{\ell}\pi_{k}^{-1}\norm{f}\E\left|\E\left[\sum_{g\in G_{k}}\left(X_{g}^{\ell}-p_{\ell}\right)\left|\mathcal{G}^{(1)}\right.\right]\right| \notag.
\end{align*}

Recalling the definition of $T_{g}^{k}$ from \eqref{eq:def_YT}, observe that as $g$ ranges over $G_{k}$, $Y_{g}^{k}$ varies but $T_{g}^{k}$ does not.  
By construction, $V^{(1)}$ is measurable with respect to the $\sigma$-field generated by $T_{g}^{k}$ and 
$\{\eps(g_{j})\}_{g\in G_{\ell}\setminus G_{k},j\geq0}$. We set $\mathcal{G}^{(1)}$ equal to this $\sigma$-field.
Since $(X_{g},T_{g}^{k})$ is independent of $\{\eps(g_{j})\}_{g\in G_{\ell}\setminus G_{k},j\geq0}$ for each $g\in G_{k}$, 
Proposition \ref{cond_exp_ind} implies that
\begin{align*}
\E\left[\sum_{g\in G_{k}}\left(X_{g}^{\ell}-p_{\ell}\right)\left|\mathcal{G}^{(1)}\right.\right]
& =\E\left[\sum_{g\in G_{k}}\left(X_{g}^{\ell}-p_{\ell}\right)\left|T_{g}^{k}\right.\right]
=\sum_{g\in G_{k}}\left(\E\left[X_{g}^{\ell}\left|T_{g}^{k}\right.\right]-p_{\ell}\right).
\end{align*}
It follows that
\begin{align}
 \pi_{\ell}\pi_{k}^{-1}\norm{f}
\E\left|\E\left[\sum_{g\in G_{k}}(X_{g}^{\ell}-p_{\ell})\left|\mathcal{G}^{(1)}\right.\right]\right|
 & = \pi_{\ell}\pi_{k}^{-1}\norm{f}\E\bigg|\sum_{g\in G_{k}}
\left(\E\left[X_{g}^{\ell}\big|T_{g}^{k}\right]-p_{\ell}\right)\bigg| \notag  \\
& \leq \pi_{\ell}\pi_{k}^{-1}\norm{f}\sum_{g\in G_{k}} \E\left|\left(\E\left[X_{g}^{\ell}\big|T_{g}^{k}\right]-p_{\ell}\right)\right|   \notag  \\
& = \pi_{\ell}\norm{f} \E\left|\left(\E\left[X_{g}^{\ell}\big|T_{g}^{k}\right]-p_{\ell}\right)\right|, \label{eq:room_for_improvement}
\end{align}
so 
\begin{align*}
b_3 & \leq \pi_{\ell}\norm{f}\E \abs{\E\left[X_{g}^{\ell}\big|T_{g}^{k}\right]-p_{\ell}} \\
 & \leq \pi_{\ell}\left(1\wedge\lambda(\ell)^{-\frac{1}{2}}\right)\E\abs{\P\left\{Y_g^{k} \in I_{\ell} - T_{g}^{k} \big| T_{g}^{k}\right\} - p_{\ell}}.
\end{align*}

Let $\nu_k$ denote the distribution of  $T_{g}^{k}$ and let $\widetilde T_{g}^{k}$ be an independent copy of $T_{g}^{k}$. 
Observe that $T_g^k \in [-a_k,a_k]$ for all $g \in G_{\ell}$, where 
\begin{align*}
a_k := \sum_{i=k+1}^{\infty}\alpha_{i}.
\end{align*}
We see that
\begin{align*}
& \E\abs{\P\left\{Y_g^{k} \in I_{\ell} - T_{g}^{k} \big| T_{g}^{k}\right\} - p_{\ell}}\\
& \quad = \int_{-a_k}^{a_k} \abs{\P\left\{Y_g^{k} \in I_{\ell} - t\right\}
 - \int_{-a_k}^{a_k} \E\left[ X_g^{\ell} \big| \widetilde T_{g}^{k} = s \right] d\nu_k(s)}d\nu_k(t)\\
& \quad\enskip\leq \int_{-a_k}^{a_k}\int_{-a_k}^{a_k} \abs{\P\left\{Y_g^{k} \in I_{\ell} - t\right\}
 - \P\left\{Y_g^{k} \in I_{\ell} - s\right\}} d\nu_k(s) d\nu_k(t)\\ 
& \qquad= \E_{T_{g}^{k}, \widetilde T_{g}^{k}} \abs{\P\left\{Y_g^{k} \in \{I_{\ell} - T_{g}^{k}\} 
\setminus \{I_{\ell} - \widetilde T_{g}^{k}\}\right\} -  \P\left\{Y_g^{k} \in \{I_{\ell} - \widetilde T_{g}^{k}\} \setminus \{I_{\ell} - T_{g}^{k}\}\right\}}.
\end{align*} 
Because $T_{g}^{k} \in [-a_k, a_k]$, we have
\begin{align*}
&\{ I_{\ell} - T_{g}^{k}\} \supseteq \left[ t_0 - \frac{c}{2\pi_{\ell}}  + a_k, t_0 + \frac{c}{2\pi_{\ell}} - a_k \right],\\
& \{ I_{\ell} - T_{g}^{k}\} \subseteq \left[ t_0 - \frac{c}{2\pi_{\ell}}  - a_k, t_0 + \frac{c}{2\pi_{\ell}} + a_k \right],
\end{align*}
hence
\begin{multline*}
\{I_{\ell} - T_{g}^{k}\} \setminus \{ I_{\ell} - \widetilde T_{g}^{k}\}\\
\subseteq \left[ t_0 - \frac{c}{2\pi_{\ell}}  - a_k, t_0 + \frac{c}{2\pi_{\ell}} + a_k \right]
\setminus \left[ t_0 - \frac{c}{2\pi_{\ell}}  + a_k, t_0 + \frac{c}{2\pi_{\ell}} - a_k \right] \\
=\left[ t_0 - \frac{c}{2\pi_{\ell}}  - a_k, t_0 - \frac{c}{2\pi_{\ell}} + a_k \right] 
\bigcup \left[ t_0 + \frac{c}{2\pi_{\ell}}  - a_k, t_0 + \frac{c}{2\pi_{\ell}} + a_k \right].
\end{multline*}
Therefore, 
\begin{align*}
&\E_{T_{g}^{k}, \widetilde T_{g}^{k}} \abs{\P\left\{Y_g^{k} \in \{I_{\ell} - T_{g}^{k}\} 
\setminus \{I_{\ell} - \widetilde T_{g}^{k}\}\right\} - \P\left\{Y_g^{k} \in \{I_{\ell}
 - \widetilde T_{g}^{k}\} \setminus \{I_{\ell} - T_{g}^{k}\}\right\}} \\ 
& \enskip \leq\ 2\P\left\{Y_g^{k} \in \left[ t_0 - \frac{c}{2\pi_{\ell}}  - a_k, t_0 - \frac{c}{2\pi_{\ell}} + a_k \right] 
\bigcup \left[ t_0 + \frac{c}{2\pi_{\ell}}  - a_k, t_0 + \frac{c}{2\pi_{\ell}} + a_k \right]\right\}  \\
& \enskip \leq\ 2\P\left\{U_g \in \left[ t_0 - \frac{c}{2\pi_{\ell}}  - 2a_k, t_0 - \frac{c}{2\pi_{\ell}} + 2a_k \right] 
\bigcup \left[ t_0 + \frac{c}{2\pi_{\ell}}  - 2a_k, t_0 + \frac{c}{2\pi_{\ell}} + 2a_k \right]\right\} \\
& \enskip=\ 8a_k \left(\frac{1}{4a_k} \int_{t_0 - \frac{c}{2\pi_{\ell}}  - 2a_k}^{t_0 - \frac{c}{2\pi_{\ell}}  
+ 2a_k} \eta(x) dx + \frac{1}{4a_k}\int_{t_0 + \frac{c}{2\pi_{\ell}}  - 2a_k}^{t_0 + \frac{c}{2\pi_{\ell}}  + 2a_k} \eta(x) dx \right)  \\ 
& \enskip \leq\ 16\norm{\eta} K \pi_{k+1}^{-(1+\gamma)} \\
& \enskip \leq\ 16\frac{1}{\alpha_0} \norm{\eta_{\eps}} K \pi_{k+1}^{-(1+\gamma)},
\end{align*}
where we recall that $\eta(x)$ is the density of $U_g$ and is bounded by 
$\frac{1}{\alpha_0} \norm{\eta_{\eps}}$. 
This concludes the proof of the lemma. \qed

$ $

The selection of the dependency neighborhoods in Lemma~\ref{lem:dep_nhoods} is based on the upper bounds derived in the above proof.  
Although the bounds on $b_1$ and $b_2$ are of the correct order of magnitude, it is possible that our use of the triangle inequality in 
\eqref{eq:room_for_improvement} may have considerably affected the quality of the bound on $b_3$. 
Ideally, we would like to keep the summation inside the absolute value in order to capitalize on the resulting cancellation, 
but it is not clear how to estimate the left hand side of \eqref{eq:room_for_improvement} directly. 
In the original statement of Theorem \ref{thm:AGG}, the triangle inequality is also employed \cite{AGG}, 
but the choice of dependency neighborhoods in our case potentially obviates this step. 
This observation may be useful in other Poisson approximation problems in which there is long range 
dependence and the dependency neighborhoods partition the collection of indicators.

\section{Appendix}
\label{sec:appendix}
\setcounter{equation}{0}

The following propositions are easy results about conditional expectation which are needed in the preceding analysis. 
In both cases, we assume that the relevant random variables are real-valued and defined on an underlying probability 
space denoted $(\Omega,\mathcal{F},P)$.

\begin{proposition}
\label{cond_exp_ind}
If $X\in L^{1}$ and $(X,Y)$ is independent of $(Z_{1},\ldots,Z_{m})$,
then 
\[
\E\left[X\left|Y,Z_{1},\ldots,Z_{m}\right.\right]=\E\left[X\left|Y\right.\right].
\]
\end{proposition}

\begin{proof}
We first observe that $\E\left[X\left|Y\right.\right]\in\sigma(Y)\subseteq\sigma(Y,Z_{1},\ldots,Z_{m})$. 
Now let $\mathcal{P}=\{A\cap B:A\in\sigma(Y),B\in\sigma(Z_{1},\ldots,Z_{m})\}$.
Then $\mathcal{P}$ is a $\pi$-system which generates $\sigma(Y,Z_{1},\ldots,Z_{m})$. 
The independence assumption ensures that for any $A\cap B\in\mathcal{P}$,
\begin{align*}
\int_{A\cap B}\E\left[X\left|Y\right.\right]dP & =\int_{A}1_{B}\E\left[X\left|Y\right.\right]dP=P(B)\int_{A}\E\left[X\left|Y\right.\right]dP\\
 & =P(B)\int_{A}XdP=\int_{A}1_{B}XdP=\int_{A\cap B}XdP.
\end{align*}

Since the collection of events $S$ for which $\int_{S}\E\left[X\left|Y\right.\right]dP=\int_{S}XdP$ is a $\lambda$-system 
containing $\mathcal{P}$ (by linearity and dominated convergence), the result follows from Dynkin's theorem.
\end{proof}

\begin{proposition}
\label{cond_exp_fcn}
If $X,Y,Z$ are independent random variables with $X\stackrel{d}{=}Y$, then
\[
\E\left[f\left(X+Z\right)f\left(Y+Z\right)\right]=\E\left[\E\left[\left.f\left(X+Z\right)\right|Z\right]^{2}\right]
\]
for any integrable or nonnegative function $f$.
\end{proposition}

\begin{proof}
Let $\mu$ denote the common distribution of $X$ and $Y$, and let $\nu$ denote the distribution of $Z$.
By independence, $\left(X,Y,Z\right)$ has distribution $\mu\times\mu\times\nu$.

Now for every $F\in\sigma(Z)$, there is a Borel set $A$ such that $F=\left\{ Z\in A\right\} $, so the 
change of variables formula and Fubini-Tonelli show that 
\begin{align*}
\int_{F}f\left(X+Z\right)f\left(Y+Z\right)dP & =\int_{A}\int_{\R}\int_{\R}f(x+z)f(y+z)d\mu(x)d\mu(y)d\nu(z)\\
 & =\int_{A}\left(\int_{\R}f(x+z)d\mu(x)\right)\left(\int_{\R}f(y+z)d\mu(y)\right)d\nu(z)\\
 & =\int_{A}\left(\int_{\R}f(x+z)d\mu(x)\right)^{2}d\nu(z)\\
 & =\int_{A}\E\left[f\left(X+z\right)\right]^{2}d\nu(z)=\int_{F}g(Z)^{2}dP
\end{align*}
with $g(z)=\E\left[f\left(X+z\right)\right]$.

On the other hand, $g(Z)=\E\left[\left.f\left(X+Z\right)\right|Z\right]$ since 
\begin{align*}
\int_{F}f(X+Z)dP & =\int_{A}\int_{\R}f(x+z)d\mu(x)d\nu(z)\\
 & =\int_{A}\E\left[f(X+z)\right]d\nu(z)=\int_{F}g(Z)dP.
\end{align*}
Therefore, $\E\left[\left.f\left(X+Z\right)f\left(Y+Z\right)\right|Z\right]=\E\left[\left.f\left(X+Z\right)\right|Z\right]^{2}$
and the result follows upon taking expectations.
\end{proof}

\bibliographystyle{abbrv}
\bibliography{Poisson}

\def\polhk#1{\setbox0=\hbox{#1}{\ooalign{\hidewidth
  \lower1.5ex\hbox{`}\hidewidth\crcr\unhbox0}}}
  \def\polhk#1{\setbox0=\hbox{#1}{\ooalign{\hidewidth
  \lower1.5ex\hbox{`}\hidewidth\crcr\unhbox0}}}
\begin{thebibliography}{10}

\bibitem{AdJod}
J.~A. Adell and P.~Jodr{\'a}.
\newblock Exact {K}olmogorov and total variation distances between some
  familiar discrete distributions.
\newblock {\em J. Inequal. Appl.}, pages Art. ID 64307, 8, 2006.
\newblock \href{http://www.ams.org/mathscinet-getitem?mr=2221228}{MR2221228}.

\bibitem{AizMol}
M.~Aizenman and S.~Molchanov.
\newblock Localization at large disorder and at extreme energies: an elementary
  derivation.
\newblock {\em Comm. Math. Phys.}, 157(2):245--278, 1993.
\newblock \href{http://www.ams.org/mathscinet-getitem?mr=1244867}{MR1244867}.

\bibitem{AlbeverioKarwowski}
S.~Albeverio and W.~Karwowski.
\newblock A random walk on {$p$}-adics---the generator and its spectrum.
\newblock {\em Stochastic Process. Appl.}, 53(1):1--22, 1994.
\newblock \href{http://www.ams.org/mathscinet-getitem?mr=1290704}{MR1290704}.

\bibitem{AGG}
R.~Arratia, L.~Goldstein, and L.~Gordon.
\newblock Two moments suffice for {P}oisson approximations: the {C}hen-{S}tein
  method.
\newblock {\em Ann. Probab.}, 17(1):9--25, 1989.
\newblock \href{http://www.ams.org/mathscinet-getitem?mr=972770}{MR972770}.

\bibitem{AGG1}
R.~Arratia, L.~Goldstein, and L.~Gordon.
\newblock Poisson approximation and the {C}hen-{S}tein method. {W}ith comments
  and a rejoinder by the authors.
\newblock {\em Statist. Sci.}, 5(4):403--434, 1990.
\newblock \href{http://www.ams.org/mathscinet-getitem?mr=1092983}{MR1092983}.

\bibitem{BHJ}
A.~D. Barbour, L.~Holst, and S.~Janson.
\newblock {\em Poisson approximation}, volume~2 of {\em Oxford Studies in
  Probability}.
\newblock The Clarendon Press, Oxford University Press, New York, 1992.
\newblock \href{http://www.ams.org/mathscinet-getitem?mr=1163825}{MR1163825}.

\bibitem{BGMS}
A.~Bendikov, A.~Grigor'yan, S.~Molchanov, and G.~Samorodnitsky.
\newblock On a class of random perturbations of the hierarchical {L}aplacian.
\newblock {\em IZV RAN Ser Mat}, 79(5):3--38, 2015.

\bibitem{BGP}
A.~Bendikov, A.~Grigor'yan, and C.~Pittet.
\newblock On a class of {M}arkov semigroups on discrete ultra-metric spaces.
\newblock {\em Potential Anal.}, 37(2):125--169, 2012.
\newblock \href{http://www.ams.org/mathscinet-getitem?mr=2944064}{MR2944064}.

\bibitem{BGPW}
A.~Bendikov, A.~Grigor'yan, C.~Pittet, and W.~Woess.
\newblock Isotropic {M}arkov semigroups on ultra-metric spaces.
\newblock {\em Uspekhi Mat. Nauk}, 69(4(418)):3--102, 2014.
\newblock \href{http://www.ams.org/mathscinet-getitem?mr=3400536}{MR3400536}.

\bibitem{BendikovKrupski}
A.~Bendikov and P.~Krupski.
\newblock On the spectrum of the hierarchical {L}aplacian.
\newblock {\em Potential Anal.}, 41(4):1247--1266, 2014.
\newblock \href{http://www.ams.org/mathscinet-getitem?mr=3269722}{MR3269722}.

\bibitem{Bovier}
A.~Bovier.
\newblock The density of states in the {A}nderson model at weak disorder: a
  renormalization group analysis of the hierarchical model.
\newblock {\em J. Statist. Phys.}, 59(3-4):745--779, 1990.
\newblock \href{http://www.ams.org/mathscinet-getitem?mr=1063180}{MR1063180}.

\bibitem{Cartier1972}
P.~Cartier.
\newblock Fonctions harmoniques sur un arbre.
\newblock In {\em Symposia {M}athematica, {V}ol. {IX} ({C}onvegno di {C}alcolo
  delle {P}robabilit\`a, {INDAM}, {R}ome, 1971)}, pages 203--270. Academic
  Press, London, 1972.
\newblock \href{http://www.ams.org/mathscinet-getitem?mr=0353467}{MR0353467}.

\bibitem{Diaconis}
S.~Chatterjee, P.~Diaconis, and E.~Meckes.
\newblock Exchangeable pairs and {P}oisson approximation.
\newblock {\em Probab. Surv.}, 2:64--106, 2005.
\newblock \href{http://www.ams.org/mathscinet-getitem?mr=2121796}{MR2121796}.

\bibitem{Figa-Tal1}
M.~Del~Muto and A.~Fig{\`a}-Talamanca.
\newblock Diffusion on locally compact ultrametric spaces.
\newblock {\em Expo. Math.}, 22(3):197--211, 2004.
\newblock \href{http://www.ams.org/mathscinet-getitem?mr=2069670}{MR2069670}.

\bibitem{Figa-Tal2}
M.~Del~Muto and A.~Fig{\`a}-Talamanca.
\newblock Anisotropic diffusion on totally disconnected abelian groups.
\newblock {\em Pacific J. Math.}, 225(2):221--229, 2006.
\newblock \href{http://www.ams.org/mathscinet-getitem?mr=2233733}{MR2233733}.

\bibitem{Dur}
R.~Durrett.
\newblock {\em Probability: theory and examples}.
\newblock Cambridge Series in Statistical and Probabilistic Mathematics.
  Cambridge University Press, Cambridge, fourth edition, 2010.
\newblock \href{http://www.ams.org/mathscinet-getitem?mr=2722836}{MR2722836}.

\bibitem{Dyson1}
F.~J. Dyson.
\newblock The dynamics of a disordered linear chain.
\newblock {\em Physical Rev. (2)}, 92:1331--1338, 1953.
\newblock \href{http://www.ams.org/mathscinet-getitem?mr=0059210}{MR0059210}.

\bibitem{Dyson2}
F.~J. Dyson.
\newblock Existence of a phase-transition in a one-dimensional {I}sing
  ferromagnet.
\newblock {\em Comm. Math. Phys.}, 12(2):91--107, 1969.
\newblock \href{http://www.ams.org/mathscinet-getitem?mr=0436850}{MR0436850}.

\bibitem{GrahamMcGehee}
C.~C. Graham and O.~C. McGehee.
\newblock {\em Essays in commutative harmonic analysis}, volume 238 of {\em
  Grundlehren der Mathematischen Wissenschaften [Fundamental Principles of
  Mathematical Science]}.
\newblock Springer-Verlag, New York-Berlin, 1979.
\newblock \href{http://www.ams.org/mathscinet-getitem?mr=550606}{MR550606}.

\bibitem{IbragimovLinnik}
I.~A. Ibragimov and Y.~V. Linnik.
\newblock {\em Independent and stationary sequences of random variables}.
\newblock Wolters-Noordhoff Publishing, Groningen, 1971.
\newblock \href{http://www.ams.org/mathscinet-getitem?mr=0322926}{MR0322926}.

\bibitem{Kochubey20004}
A.~N. Kochubei.
\newblock {\em Pseudo-differential equations and stochastics over
  non-{A}rchimedean fields}, volume 244 of {\em Monographs and Textbooks in
  Pure and Applied Mathematics}.
\newblock Marcel Dekker, Inc., New York, 2001.
\newblock \href{http://www.ams.org/mathscinet-getitem?mr=1848777}{MR1848777}.

\bibitem{Kozyrev}
S.~V. Kozyrev.
\newblock Wavelets and spectral analysis of ultrametric pseudodifferential
  operators.
\newblock {\em Mat. Sb.}, 198(1):103--126, 2007.
\newblock \href{http://www.ams.org/mathscinet-getitem?mr=2330687}{MR2330687}.

\bibitem{Kvitchevski1}
E.~Kritchevski.
\newblock Hierarchical {A}nderson model.
\newblock In {\em Probability and mathematical physics}, volume~42 of {\em CRM
  Proc. Lecture Notes}, pages 309--322. Amer. Math. Soc., Providence, RI, 2007.
\newblock \href{http://www.ams.org/mathscinet-getitem?mr=2352276}{MR2352276}.

\bibitem{Kvitchevski2}
E.~Kritchevski.
\newblock Spectral localization in the hierarchical {A}nderson model.
\newblock {\em Proc. Amer. Math. Soc.}, 135(5):1431--1440 (electronic), 2007.
\newblock \href{http://www.ams.org/mathscinet-getitem?mr=2276652}{MR2276652}.

\bibitem{Kvitchevski3}
E.~Kritchevski.
\newblock Poisson statistics of eigenvalues in the hierarchical {A}nderson
  model.
\newblock {\em Ann. Henri Poincar\'e}, 9(4):685--709, 2008.
\newblock \href{http://www.ams.org/mathscinet-getitem?mr=2413200}{MR2413200}.

\bibitem{Krutikov1}
D.~Krutikov.
\newblock On an essential spectrum of the random {$p$}-adic
  {S}chr\"odinger-type operator in the {A}nderson model.
\newblock {\em Lett. Math. Phys.}, 57(2):83--86, 2001.
\newblock \href{http://www.ams.org/mathscinet-getitem?mr=1856901}{MR1856901}.

\bibitem{Krutikov2}
D.~Krutikov.
\newblock Spectra of {$p$}-adic {S}chr\"odinger-type operators with random
  radial potentials.
\newblock {\em J. Phys. A}, 36(15):4433--4443, 2003.
\newblock \href{http://www.ams.org/mathscinet-getitem?mr=1984513}{MR1984513}.

\bibitem{LeadbetterLingremRootzen}
M.~R. Leadbetter, G.~Lindgren, and H.~Rootz{\'e}n.
\newblock {\em Extremes and related properties of random sequences and
  processes}.
\newblock Springer Series in Statistics. Springer-Verlag, New York-Berlin,
  1983.
\newblock \href{http://www.ams.org/mathscinet-getitem?mr=691492}{MR691492}.

\bibitem{Lukacs}
E.~Luka{\v{c}}.
\newblock {\em Kharakteristicheskie funktsii}.
\newblock ``Nauka'', Moscow, 1979.
\newblock \href{http://www.ams.org/mathscinet-getitem?mr=563121}{MR563121}.

\bibitem{Minami}
N.~Minami.
\newblock Local fluctuation of the spectrum of a multidimensional {A}nderson
  tight binding model.
\newblock {\em Comm. Math. Phys.}, 177(3):709--725, 1996.
\newblock \href{http://www.ams.org/mathscinet-getitem?mr=1385082}{MR1385082}.

\bibitem{Molchanov}
S.~Molchanov.
\newblock Hierarchical random matrices and operators. {A}pplication to
  {A}nderson model.
\newblock In {\em Multidimensional statistical analysis and theory of random
  matrices ({B}owling {G}reen, {OH}, 1996)}, pages 179--194. VSP, Utrecht,
  1996.
\newblock \href{http://www.ams.org/mathscinet-getitem?mr=1463464}{MR1463464}.

\bibitem{PearsonBellisard}
J.~Pearson and J.~Bellissard.
\newblock Noncommutative {R}iemannian geometry and diffusion on ultrametric
  {C}antor sets.
\newblock {\em J. Noncommut. Geom.}, 3(3):447--480, 2009.
\newblock \href{http://www.ams.org/mathscinet-getitem?mr=2511637}{MR2511637}.

\bibitem{PersSchlagSolomyak}
Y.~Peres, W.~Schlag, and B.~Solomyak.
\newblock Sixty years of {B}ernoulli convolutions.
\newblock In {\em Fractal geometry and stochastics, {II}
  ({G}reifswald/{K}oserow, 1998)}, volume~46 of {\em Progr. Probab.}, pages
  39--65. Birkh\"auser, Basel, 2000.
\newblock \href{http://www.ams.org/mathscinet-getitem?mr=1785620}{MR1785620}.

\bibitem{PersSolomyak}
Y.~Peres and B.~Solomyak.
\newblock Absolute continuity of {B}ernoulli convolutions, a simple proof.
\newblock {\em Math. Res. Lett.}, 3(2):231--239, 1996.
\newblock \href{http://www.ams.org/mathscinet-getitem?mr=1386842}{MR1386842}.

\bibitem{Ross}
N.~Ross.
\newblock Fundamentals of {S}tein's method.
\newblock {\em Probab. Surv.}, 8:210--293, 2011.
\newblock \href{http://www.ams.org/mathscinet-getitem?mr=2861132}{MR2861132}.

\bibitem{Solomyak1}
B.~Solomyak.
\newblock On the random series {$\sum\pm\lambda\sp n$} (an {E}rd{\H o}s
  problem).
\newblock {\em Ann. of Math. (2)}, 142(3):611--625, 1995.
\newblock \href{http://www.ams.org/mathscinet-getitem?mr=1356783}{MR1356783}.

\bibitem{Solomyak}
B.~Solomyak.
\newblock Notes on {B}ernoulli convolutions.
\newblock In {\em Fractal geometry and applications: a jubilee of {B}eno\^\i t
  {M}andelbrot. {P}art 1}, volume~72 of {\em Proc. Sympos. Pure Math.}, pages
  207--230. Amer. Math. Soc., Providence, RI, 2004.
\newblock \href{http://www.ams.org/mathscinet-getitem?mr=2112107}{MR2112107}.

\bibitem{Taibleson75}
M.~H. Taibleson.
\newblock {\em Fourier analysis on local fields}.
\newblock Princeton University Press, Princeton, N.J.; University of Tokyo
  Press, Tokyo, 1975.
\newblock \href{http://www.ams.org/mathscinet-getitem?mr=0487295}{MR0487295}.

\bibitem{Vladimirov}
V.~S. Vladimirov.
\newblock Generalized functions over the field of {$p$}-adic numbers.
\newblock {\em Uspekhi Mat. Nauk}, 43(5(263)):17--53, 239, 1988.
\newblock \href{http://www.ams.org/mathscinet-getitem?mr=971464}{MR971464}.

\bibitem{VladimirovVolovich}
V.~S. Vladimirov and I.~V. Volovich.
\newblock {$p$}-adic {S}chr\"odinger-type equation.
\newblock {\em Lett. Math. Phys.}, 18(1):43--53, 1989.
\newblock \href{http://www.ams.org/mathscinet-getitem?mr=1009522}{MR1009522}.

\bibitem{Vladimirov94}
V.~S. Vladimirov, I.~V. Volovich, and E.~I. Zelenov.
\newblock {\em {$p$}-adic analysis and mathematical physics}, volume~1 of {\em
  Series on Soviet and East European Mathematics}.
\newblock World Scientific Publishing Co., Inc., River Edge, NJ, 1994.
\newblock \href{http://www.ams.org/mathscinet-getitem?mr=1288093}{MR1288093}.

\bibitem{Woess}
W.~Woess.
\newblock {\em Denumerable {M}arkov chains. Generating functions, boundary
  theory, random walks on trees}.
\newblock EMS Textbooks in Mathematics. European Mathematical Society (EMS),
  Z\"urich, 2009.
\newblock \href{http://www.ams.org/mathscinet-getitem?mr=2548569}{MR2548569}.

\end{thebibliography}

\end{document}